\documentclass[12pt,a4paper]{amsart}
\usepackage{amsmath, amssymb, mathtools, amsthm}
\usepackage{enumerate}
\usepackage{ wasysym}
\newcommand{\A}{\mathcal{A}}

\newcommand{\C}{\mathbb{C}}

\newtheorem{theorem}{Theorem}[section]
\newtheorem{remark}{Remark}[section]
\newtheorem{lemma}{Lemma}[section]
\DeclareMathOperator{\RE}{Re}

\begin{document}

\title[ Radius Constants]{Radii Constants  for Functions with Fixed Second Coefficient}

\author[ Shalini Rana, Om Ahuja, Naveen Kumar Jain]{ Shalini Rana, Om Ahuja, Naveen Kumar Jain}

\address{Department of Mathematics\\ University of Delhi, Delhi--110 007, India}
\email{shalinirana3010@gmail.com}

\address{	Department of Mathematical Sciences,\\
	Kent State University, Ohio, 44021, U.S.A }
	\email{oahuja@kent.edu}

\address{Department of Mathematics\\
	Aryabhatta College,  Delhi, 110021, India}
	\email{naveenjain@aryabhattacollege.ac.in}

\begin{abstract}
We introduce three classes of analytic functions with fixed second coefficient which are defined using the class $\mathcal{P}$ of analytic functions with positive real part. The objective is to determine radii such that the three classes are contained in various subclasses of starlike functions. The radii estimated in the present investigation are better than the radii obtained earlier. Furthermore, connections with previous known results are shown.
\end{abstract}



\maketitle

\section{Introduction}
Let $\Delta$ denote the open unit disc in $\mathbb{C}$ and  the class $\A$ be defined as collection of all analytic functions in $\Delta$ satisfying $f(0)=1$ and $f'(0)=1$. The class $\mathcal{S}$ is defined to be collection of  univalent functions in class $\A$.
  The well known Bieberbach theorem states that for a univalent function $ f(z)=z+ a_2 z^{2}+...$ the bound on second coefficient, that is $|a_2|\leq2$, plays an important role in the study of  univalent function theory. This bound attracted the interest of many mathematicians which led to the investigation of the class $A_b$ consisting of the  functions of the form $ f(z)=z+ a_2 z^{2}+...$, $|a_2|=b$ for a fixed  $b$ with $0 \le b \le 1$. For $n \in \mathbb{N}$ and $0 \le b \le 1$, let  $\A_{nb}$  be the class  of analytic functions of the form $ f(z)=z+ nb z^{2}+...$ for $z \in \Delta$ such that $\A_b:=\A_{1b}$. The study of  class $A_b$ was initiated as early as 1920 by  Gronwall \cite{Gronwall}. He determined the growth and distortion estimates for the class of univalent functions with fixed second coefficients. In 2011, Ali \emph{et al.} \cite{ANV} obtained various results for the class of functions with fixed second coefficients by applying the theory of second order differential subordination. Later, Lee \emph{et al.} \cite{lee} investigated certain applications of differential subordination for such functions. Kumar et al. \cite{kumar1} determined the best possible estimates on initial coefficients of Ma-Minda type univalent functions. Ali \emph{et al.} \cite{kum}  obtained sharp radii of starlikeness for  certain classes of functions with fixed second coefficients. A survey on functions with fixed initial coefficient can be found in \cite{ali1}.

Let $f$ and $h$ be analytic functions in $\Delta$. Then $f$ is subordinate to $h$, that is $f\prec h$ if $f(z)=h(w(z))$ for some analytic function $w$, with $w(0)=0$ and $|w(z)|<1$. In particular, if $h\in\mathcal{S}$, then $f\prec h$ provided $f(0)=h(0)$ and $f(\Delta)\subset h(\Delta)$.   Several well known subclasses of starlike  functions were characterized by subordination of $zf'/f$ to some function in $\mathcal{P}$. Using a function $\psi$, univalent in unit disc $\Delta$ with $\RE(\psi(z))>0$, starlike with respect to $\psi(0)=1$, symmetric about the real axis and $\psi'(0)>0$, Ma and Minda \cite{MM} gave a unified treatment for several subclasses of starlike  functions by studying the class $\mathcal{S}^*(\psi)=$\{ $f \in \A: zf'(z)/f(z) \prec \psi(z)$\}.  Recently, Anand \emph{et al.} \cite{anand} also obtained some results for  a class of analytic functions defined using the function $\psi$. Several authors considered the class $\mathcal{S}^*(\psi)$ for various choices of the function $\psi.$ For $-1<B<A\le 1,$ let $\psi(z)=(1+Az)/(1+Bz)$ be a convex function whose image is symmetric with
respect to the real axis. For this function $\psi$, the class $\mathcal{S}^*(\psi)$ reduces  to the class $\mathcal{S}^*[A,B]$ \cite{Jan} (see also \cite{sri1})  of Janowski starlike functions. Other well-known choices for $\psi(z)$ include $\sqrt{1+z}$, $\exp{z}$, $1+(4/3)z+(2/3)z^2$, $1+sin(z)$, $z+ \sqrt{1+z^2}$, $1+z-z^3/3$, $2/(1+e^{-z})$. For a brief survey of these classes, reader may refer to \cite{goodman, sri}.\\

Consider the class $\mathcal{P}(\alpha)$ of analytic functions $p(z)=1+b_1 z+b_2 z^2+...$ satisfying $\RE\{p(z)\}>\alpha$ where $0 \le \alpha <1$. It can be noted that the class $\mathcal{P}=\mathcal{P}(0)$. For $p \in \mathcal{P}(\alpha)$ \cite{nehari}, we have $ |b_1| \le 2(1-\alpha).$
Denote by $\mathcal{P}_b(\alpha)$ the subclass of $\mathcal{P}(\alpha)$ consisting of functions of the form  $p(z)=1+2b(1-\alpha)z+...$ for $b \in [0,1]$ such that $\mathcal{P}_b = \mathcal{P}_b(0)$.
For any two subclasses $\mathcal{M}$ and $\mathcal{N}$ of $\A$, the $\mathcal{N}$ radius for the class $\mathcal{M}$ is the largest number $\rho \in (0,1)$ such that $r^{-1}f(rz) \in \mathcal{N},$ for all $f \in \mathcal{M}$ and for $0<r<\rho.$
Among several studies available on radius problems, many results involving  ratio between two classes of functions, where one of them belong to some particular subclasses of $\A$   has been a center of major focus in geometric function theory and can be seen in \cite{ratti,ratti1,gregor,gregor1,gregor2}.
Recently, Lecko and Ravichandran\cite{lecko} estimated certain best possible radii for the classes of function $g\in \mathcal{A}$  satisfying the  conditions $(i)$  $g/h\in\mathcal{P}$ where $h/(zp)\in\mathcal{P}$ or  $h/(zp)\in\mathcal{P}(1/2)$ $(ii)$ $g/(zp)\in\mathcal{P}$.
Motivated by above work, we define three classes of functions by making use of the classes $\A_{6b}$, $\A_{4c}$ and $\mathcal{P}$ as follows:
\begin{equation*}
H^1_{b,c}=\{f \in \A_{6b}: \frac{f}{g} \in \mathcal{P} \  and \  \frac{g}{zp} \in \mathcal{P} \  where \  g \in \A_{4c}, \  p \in \mathcal{P}  \}
\end{equation*}
\begin{equation*}
H^2_{b,c}=\{f \in \A_{5b}: \frac{f}{g} \in \mathcal{P} \  and \  \frac{g}{zp} \in \mathcal{P}(1/2) \  where \  g \in \A_{3c}, \  p \in \mathcal{P}  \}\end{equation*}
and
\begin{equation*}
H^3_b=\{f \in \A_{4b}:  \frac{f}{zp} \in \mathcal{P} \ where \  p \in \mathcal{P}  \}
\end{equation*}
where $b \in [0,1]$ and $c \in [0,1]$.

The results presented in this paper are nice extensions of radii estimates results in  \cite{lecko} along with some improved radii. It includes  radii estimates for functions in the classes $H^1_{b,c}$, $H^2_{b,c}$ and $H^3_b$ to belong to several subclasses of normalized analytic functions $\mathcal{A}$ like starlike functions of order $\alpha$, starlike functions associated with lemniscate of Bernoulli, parabolic starlike functions, exponential function,   cardioid,  sine function, lune,  a particular rational function, nephroid  and  modified sigmoid function.

\section{Analysis and mapping of $zf'(z)/f(z)$ for $H^1_{b,c}$, $H^2_{b,c}$ and $H^3_b$}
\noindent
In this section, we investigate extremal functions for all three classes $H^1_{b,c}$, $H^2_{b,c}$ and $H^3_b$ which demonstrates the fact that the classes are non empty. Furthermore, we obtain the disc centered at 1, containing the image of the disc $\Delta$
under the mapping  $zf'/f$ where $f$ belongs to each of these classes.
We begin by stating the following lemmas by McCarty:
\begin{lemma}\label{lm1} \cite{carty}
Let $b \in [0,1]$ and $0 \le \alpha <1$. If $p \in \mathcal{P}_b(\alpha)$, then for $|z|=r<1$,
\begin{equation*}
    \left | \frac{zp'(z)}{p(z)}\right | \le \frac{2(1-\alpha)r}{1-r^2} \frac{b r^2+2r+b}{(1-2\alpha)r^2+ 2b(1-\alpha) r+1}.
\end{equation*}
\end{lemma}
\begin{lemma}\cite{car}\label{lm2}
Let $b \in [0,1]$ and $0 \le \alpha <1$. If $p \in \mathcal{P}_b(\alpha)$, then for $|z|=r<1$,
\begin{equation*}
    \RE\left ( \frac{zp'(z)}{p(z)}\right ) \geq
    \begin{cases} \frac{-2(1-\alpha)r}{1+2\alpha b r+(2\alpha-1)r^2} \frac{b r^2+2r+b}{r^2+ 2b r+1}\,\, \text{if} \,\, R_{\alpha}\leq R_b\\
    \frac{2\sqrt{\alpha C_1}-C_1-\alpha}{1-\alpha}\quad\quad \quad\quad\,\,\,\,\text{if}\,\,  R_{\alpha}\geq R_b
    \end{cases},
\end{equation*}
where $R_b=C_b-D_b$, $R_{\alpha}=\sqrt{\alpha C_1}$, and $r=|z|<1$;
\begin{equation}\label{eqn1}
C_b=\frac{(1+br)^2-(2\alpha-1)(b+r)^2r^2}{(1+2br+r^2)(1-r^2)}\quad \text{and}\quad D_b=\frac{2(1-\alpha)(b+r)(1+br)r}{(1+2br+r^2)(1-r^2)}
\end{equation}

\end{lemma}
\subsection{\textbf{The class $H^1_{b,c}$}}
\noindent
Let the functions $f$ and $g$ whose Taylor series expansions are given by $f(z)=z+ f_1 z^2+...$ and $g(z)=z+g_1 z^2+...$ be such that $\RE\{f/g\}>0$ and $\RE\{g/(zp)\}>0$ where $p \in \mathcal{P}$ is represented by $p(z)=1+ q z+...$ Now consider, $ {g}/{(zp)}= 1+(g_1-q)z+...$ where $|g_1-q| \le 2$ and $|q| \le 2$ which gives $|g_1| \le 4.$ Also, ${f}/{g}=1+(f_1-g_1)z+....$
where $|f_1-g_1| \le 2$ and hence $|f_1| \le 6$.
Hence, for $b , c \in [0,1]$, we consider the class involving the function $f$ and $g$ with fixed second coefficient whose Taylor series expansion is given by $f(z)=z+6b z^2+...$ and $ g(z)=z+4c z^2 +... $ such that $f \in \A_{6b}$ and $g \in \A_{4c}$.
If the function  $f \in$ $H^1_{b,c}$, then there exists an element $g \in \A_{4c}$ and $p \in \mathcal{P}$ such that $f/g \in \mathcal{P}$ and $g/(zp) \in \mathcal{P}$. Define
\begin{equation*}
    h(z)= \frac{f(z)}{g(z)}= 1+ (6b-4c)z+...\quad
\text{and} \quad    k(z)= \frac{g(z)}{zp(z)}= 1+ (4c-q)z+...
\end{equation*}
where $|6b-4c| \le 2$ and $|4c-q|\le 2$. Therefore, we observe that $h \in \mathcal{P}_{(6b-4c)/2}$, $k \in \mathcal{P}_{(4c-q)/2}$, $p \in \mathcal{P}_{q/2}$ and $f$ can be expressed as $f(z)= z p(z) h(z) k(z)$. Then a calculation shows that
\begin{equation}\label{eq0.2}
    \left| \frac{zf'(z)}{f(z)}-1 \right| \le \left| \frac{zh'(z)}{h(z)} \right| +\left| \frac{zk'(z)}{k(z)} \right| +\left| \frac{zp'(z)}{p(z)} \right|.
\end{equation}
 For $d=|6b-4c| \le 2$, $s=|4c-q| \le 2$ and $\alpha=0$, using  lemma \ref{lm1}, we get
\begin{equation}\label{eq0.3}
 \left| \frac{zh'(z)}{h(z)}\right|   \le \frac{r}{(1-r^2)} \frac{(dr^2+4r+d)}{(r^2+dr+1)},\quad
\left| \frac{zk'(z)}{k(z)}\right|   \le \frac{r}{(1-r^2)} \frac{(sr^2+4r+s)}{(r^2+sr+1)},
\end{equation}
and
\begin{equation}\label{eq0.4}
    \left| \frac{zp'(z)}{p(z)}\right|   \le \frac{r}{(1-r^2)} \frac{(q r^2+4r+q)}{(r^2+q r+1)}.
\end{equation}
By \eqref{eq0.2}, \eqref{eq0.3} and \eqref{eq0.4}, it follows that
\begin{equation}\label{eq0.5}
       \left| \frac{zf'(z)}{f(z)}-1 \right| \le \frac{ \splitfrac
       {r[(dr^2+4r+d)(r^2+sr+1)(r^2+qr+1)+(sr^2+4r+s)}{(r^2+dr+1)(r^2+qr+1)+(qr^2+4r+q)(r^2+dr+1)(r^2+sr+1)]}}
       {(1-r^2)(r^2+dr+1)(r^2+sr+1)(r^2+qr+1)}.
\end{equation}
Define the functions $f_1$, $g_1$  and $p_1:\Delta \to \mathbb{C}$ by
\begin{equation}\label{eq0.6}
f_1(z)= \frac{z(1-qz+z^2)(1-(4c-q)z+z^2)(1-(6b-4c)z+z^2)}{(1-z^2)^3},
\end{equation}
\begin{equation}\label{eq0.7}
g_1(z)= \frac{z(1-qz+z^2)(1-(4c-q)z+z^2)}{(1-z^2)^2},\quad
\text{and}\quad    p_1(z)= \frac{(1-qz+z^2)}{(1-z^2)},
\end{equation}
\[\text{where} \quad  |6b-4c| \le 2,  \quad  |4c-q| \le 2\quad \text{and}\quad |q| \le 2.\]
By \eqref{eq0.6} and \eqref{eq0.7}, we have
\begin{equation*}
\frac{f_1(z)}{g_1(z)}=\frac{(1-(6b-4c)z+z^2)}{(1-z^2)}=\frac{1+w_1(z)}{1-w_1(z)}
\end{equation*}
\begin{equation*}
\frac{g_1(z)}{z p_1(z)}=\frac{(1-(4c-q)z+z^2)}{(1-z^2)}=\frac{1+w_2(z)}{1-w_2(z)}
 \,\,\text{and}\,\,
p_1(z)=\frac{(1-q z+z^2)}{(1-z^2)}=\frac{1+w_3(z)}{1-w_3(z)}
\end{equation*}
where
\begin{equation*}
w_1(z)=\frac{z(z-(6b-4c)/2)}{(1-z(6b-4c)/2)},\,
w_2(z)=\frac{z(z-(4c-q)/2)}{(1-z(4c-q)/2)}
\,\,\text{and}\,\,
w_3(z)=\frac{z(z-q/2)}{(1-z q/2)}
\end{equation*}
which are analytic functions satisfying the conditions of Schwarz lemma in the unit disc $\Delta$ and hence $\RE(f_1/g_1)>0$, $\RE(g_1/(z p_1))>0$ and $\RE(p_1)>0$. Thus, $f_1/g_1 \in \mathcal{P}_{(6b-4c)/2}$, $g_1/(z p_1) \in \mathcal{P}_{(4c-q)/2}$ and $p_1 \in \mathcal{P}_{(q/2)}$. Thus the function $f_1 \in H^1_{b,c}$ and the class $H^1_{b,c}$ is non-empty.
\begin{equation}\label{eq0.8}
  F_1(z)=\frac{z(1-z^2)^3}{(1-qz+z^2)(1-(4c-q)z+z^2)(1-(6b-4c)z+z^2)}
\end{equation}
and $f_1(z)$ are extreme functions for the class $H^1_{b,c}$ provided $q \le 2, \  c \ge q/4$ and $b \ge 2c/3.$

\subsection{\textbf{The class $H^2_{b,c}$}}
\noindent
Let  $f$ and $g$ be functions with Taylor series expansions  $ f(z)=z+ f_1 z^2+... $ and $g(z)=z+g_1 z^2+...$
 such that $\RE\{f/g\}>0$ and $\RE\{g/(zp)\}> 1/2$ where $p \in \mathcal{P}$ is represented by $  p(z)=1+ q z+... $
Now consider, $g/zp = 1+(g_1-q)z+...$
where $|g_1-q| \le 1$ (lemma 2, pg 33 et al \cite{owa}) and $|q| \le 2$ which gives $|g_1| \le 3.$ Also, $f/g =1+(f_1-g_1)z+....$
where $|f_1-g_1| \le 2$ and hence $|f_1| \le 5$.
Therefore, we consider the class involving the function $f$ and $g$ with fixed second coefficient whose Taylor series expansion is given by $ f(z)=z+5b z^2+...$ and $g(z)=z+3c z^2 +...$
where $b \in [0,1]$ and $c \in [0,1]$ such that $f \in \A_{5b}$ and $g \in \A_{3c}$. If the function  $f \in$ $H^2_{b,c}$ then there exists an element $g \in \A_{3c}$ and $p \in \mathcal{P}$ such that $f/g \in \mathcal{P}$ and $g/(zp) \in \mathcal{P}(1/2).$ Define
\begin{equation*}
    h(z)= \frac{f(z)}{g(z)}= 1+ (5b-3c)z+...
\end{equation*}
 and
\begin{equation*}
    k(z)= \frac{g(z)}{zp(z)}= 1+ (3c-q)z+...
\end{equation*}
  It is easy to see  that $h \in \mathcal{P}_{(5b-3c)/2}$, $k \in \mathcal{P}_{(3c-q)}$, $p \in P_{q/2}$ and $f(z)= z p(z) h(z) k(z)$. Then a calculation shows that
\begin{equation}\label{eq1.01}
     \frac{zf'(z)}{f(z)}  = 1+ \frac{zh'(z)}{h(z)} + \frac{zk'(z)}{k(z)}  + \frac{zp'(z)}{p(z)}
\end{equation}
so that
\begin{equation}\label{eq1.0}
    \left| \frac{zf'(z)}{f(z)}-1 \right| \le \left| \frac{zh'(z)}{h(z)} \right| +\left| \frac{zk'(z)}{k(z)} \right| +\left| \frac{zp'(z)}{p(z)} \right|.
\end{equation}
 Let $m=|5b-3c| \le 2$, $n=|3c-q| \le 1$ and $\alpha=0$. Using Lemma \ref{lm1} for the functions $h, p$ and $k$, we have
\begin{equation}\label{eq1.1}
 \left| \frac{zh'(z)}{h(z)}\right|   \le \frac{r}{(1-r^2)} \frac{(m r^2+4r+m)}{(r^2+m r+1)}, \,\,
\left| \frac{zk'(z)}{k(z)}\right|   \le \frac{r}{(1-r^2)} \frac{(n r^2+2r+n)}{(n r+1)},
\end{equation}
and
\begin{equation}\label{eq1.2}
    \left| \frac{zp'(z)}{p(z)}\right|   \le \frac{r}{(1-r^2)} \frac{(qr^2+4r+q)}{(r^2+qr+1)}.
\end{equation}
Inequality \eqref{eq1.0} together with \eqref{eq1.1} and \eqref{eq1.2},  gives
\begin{equation}\label{eq2.3}
 \left| \frac{zf'(z)}{f(z)}-1 \right| \le \frac{r}{(1-r^2)}\left( \frac{mr^2+4r+m}{r^2+mr+1}+\frac{nr^2+2r+n}{nr+1}+\frac{q r^2+4r+q}{r^2+q r+1}\right).
\end{equation}
Define the functions $f_2$, $g_2$ $:\Delta \to \mathbb{C}$ by
\begin{equation}\label{eq1.3}
f_2(z)= \frac{z(1-qz+z^2)(1-(3c-q)z)(1-(5b-3c)z+z^2)}{(1-z^2)^3}
\end{equation}
and
\begin{equation}\label{eq1.4}
g_2(z)= \frac{z(1-qz+z^2)(1-(3c-q)z)}{(1-z^2)^2},  \text{where}\,\, |5b-3c| \le 2\,\,\text{and}\,\, |3c-q| \le 1.
\end{equation}

It follows from \eqref{eq0.7}, \eqref{eq1.3} and \eqref{eq1.4}  that
\begin{equation*}
\frac{f_2(z)}{g_2(z)}=\frac{(1-(5b-3c)z+z^2)}{(1-z^2)}=\frac{1+w_4(z)}{1-w_4(z)}
\end{equation*}
and
\begin{equation*}
\frac{g_2(z)}{z p_1(z)}=\frac{(1-(3c-q)z)}{(1-z^2)}=\frac{1+w_5(z)}{1-w_5(z)}
\end{equation*}
where
\begin{equation*}
w_4(z)=\frac{z(z-(5b-3c)/2)}{(1-z(5b-3c)/2)}\,\, \text{and}\,\, w_5(z)= \frac{z(z-(3c-q))}{(2-(3c-q)z-z^2)}
\end{equation*}
which are Schwarz functions in the unit disc $\Delta$ and hence $\RE(f_2/g_2)>0$, $\RE(g_2/(z p_1))>1/2$ and $\RE(p_1)>0$ (as shown  for class $H^1_{b,c}$). Thus, $f_2/g_2 \in \mathcal{P}_{(5b-3c)/2}$, $g_2/(z p_1) \in \mathcal{P}_{(3c-q)}$ and $p_1 \in \mathcal{P}_{(q/2)}$. Thus the function $f_2 \in H^2_{b,c}$ and the class $H^2_{b,c}$ is non-empty.
$f_2(z)$ is an extreme functions for the class $H^2_{b,c}$ provided $q \le2$, $c \ge q/3$ and $b \ge 3c/5.$
\subsection{\textbf{The class $H^3_b$}}
\noindent
Let the functions $f$ and $g$ whose Taylor series expansions are given by $f(z)=z+ f_1 z^2+..$
be such that $\RE\{f/(zp)\}>0$ where $p \in \mathcal{P}$ is represented by $ p(z)=1+ q z+... $
Now consider, $f/(zp) = 1+(f_1-q)z+...$
where $|f_1-q| \le 2$ and $|q| \le 2$ which gives $|f_1| \le 4.$
Therefore, we consider the function $f$ with fixed second coefficient whose Taylor series expansion is given by $ f(z)=z+4b z^2+...$
where $b \in [0,1]$ such that $f \in \A_{4b}$. If the function  $f \in$ $H^3_b$ then there exists $p \in \mathcal{P}$ such that $f/(zp) \in \mathcal{P}$. Define
\begin{equation*}
    h(z)= \frac{f}{zp}(z)= 1+ (4b-q)z+...
\end{equation*}
     so that $h \in \mathcal{P}_{(4b-q)/2}$, $p \in \mathcal{P}_{q/2}$ and $f$ can be expressed as $f(z)= z p(z) h(z)$. Then a calculation shows that
\begin{equation}\label{eq2.1}
    \left| \frac{zf'(z)}{f(z)}-1 \right| \le \left| \frac{zh'(z)}{h(z)} \right| +\left| \frac{zp'(z)}{p(z)} \right|.
\end{equation}
From lemma \ref{lm1}, for $l=|4b-q| \le 2, |q|\leq2$ and $\alpha=0$, we get
\begin{equation}\label{eq2.2}
 \left| \frac{zh'(z)}{h(z)}\right|   \le \frac{r}{(1-r^2)} \frac{(l r^2+4r+l)}{(r^2+l r+1)}
\quad\text{and}\quad
    \left| \frac{zp'(z)}{p(z)}\right|   \le \frac{r}{(1-r^2)} \frac{(qr^2+4r+q)}{(r^2+qr+1)}.
\end{equation}
Using  \eqref{eq2.1} and \eqref{eq2.2}, it is easy to show that
\begin{equation}\label{eq02.2}
       \left| \frac{zf'(z)}{f(z)}-1 \right| \le \frac{
       r[(l r^2+4r+l)(qr^2+4r+q)+(qr^2+4r+q)(r^2+l r+1)]}
       {(1-r^2)(r^2+l r+1)(r^2+qr+1)}.
\end{equation}
Define the functions $f_3:\Delta \to \mathbb{C}$ by
\begin{equation}\label{eq0.10}
f_3(z)= \frac{z(1-qz+z^2)(1-(4b-q)z+z^2)}{(1-z^2)^2}, \  |4b-q| \le 2.
\end{equation}
Then, \eqref{eq2.2} together with \eqref{eq0.7} gives
\begin{equation*}
\frac{f_3(z)}{z p_1(z)}=\frac{(1-(4b-q)z+z^2)}{(1-z^2)}=\frac{1+w_6(z)}{1-w_6(z)},\,\,
\text{where}\,\,
w_6(z)=\frac{z(z-(4b-q)/2)}{(1-z(4b-q)/2)}
\end{equation*}
which is analytic function satisfying the conditions of Schwarz lemma in the unit disc $\Delta$ and hence $\RE(f_3/(z p_1))>0$ and $\RE(p_1)>0$ (shown above in class $H^1_{b,c}$). Thus, $f_3/(z p_1) \in \mathcal{P}_{(6b-q)/2}$ and $p_1 \in \mathcal{P}_{(q/2)}$. Thus, the function $f_3 \in H^3_{b}$ and the class $H^3_{b}$ is non-empty.
\begin{equation}\label{eq2.5}
F_3(z)= \frac{z(1-z^2)^2}{(1-qz+z^2)(1-(4b-q)z+z^2)},\quad |4b-q|\leq2
\end{equation}
and $f_3(z)$ are extreme functions for the class $H^3_{b}$ provided $q \le 2$ and $b \ge q/4.$
\\
Using the information in the above subsections, we investigate several radius problems associated with functions in the classes $H^1_{b,c}$, $H^2_{b,c}$ and $H^3_b$ in the next section.
\section{Radius of Starlikeness}
\noindent
In the present  section, we determine the radii of the classes $H^1_{b,c}$, $H^2_{b,c}$ and $H^3_b$ to belong to various Ma-Minda starlike classes of functions. Throughout this section we assume that $d=|6b-4c|\leq2$, $s=|4c-q|\leq2$, $m=|5b-3c|\leq2$, $n=|3c-q|\leq1$, $l=|4b-q|\leq2$.\\
\noindent
For $0 \le \alpha \le 1$, the class $\mathcal{S}^*(\alpha)=\mathcal{S}^*[1-2\alpha,-1]$=$\{f \in \A: \RE(zf'(z)/f(z))>\alpha\}$ is the class of starlike functions of order $\alpha$.
\begin{theorem}\label{th1}
The sharp $\mathcal{S}^*(\alpha)$ radii for the classes $H^1_{b,c}$,  $H^2_{b,c}$ and $H^3_{b}$ are as follows:
\begin{enumerate}[(i)]
    \item For the class $H^1_{b,c}$, the sharp $\mathcal{S}^*(\alpha)$ radius $\rho_1\in(0, 1)$ is the smallest  root  of the equation $x_1(r)=0$, where\\
    $x_1(r)=\alpha-1+\alpha (d+s+q) r+(10+2 \alpha+d s+\alpha d s+(1+\alpha) (d+s) q) r^2+((10+\alpha) (s+q)+d (10+\alpha+(2+\alpha) s q)) r^3+8 (3+s q+d (s+q)) r^4+(-(-12+\alpha) (s+q)-d (-12+\alpha+(-4+\alpha) s q)) r^5+(14+3 s q+3 d (s+q)-\alpha (2+s q+d (s+q))) r^6-(-2+\alpha) (d+s+q) r^7+(1-\alpha) r^8$.
    \item For the class $H^2_{b,c}$, the sharp  $\mathcal{S}^*(\alpha)$ radius $\rho_2\in(0, 1)$ is the smallest  root  of the equation $x_2(r)=0$, where\\
    \[x_2(r)=-\frac{4}{1-r^2}-\frac{1}{1+n r}+\frac{2+2 n r}{1+2 n r+r^2}+\frac{2+m r}{1+r (m+r)}+\frac{2+q r}{1+r (q+r)}-\alpha.\]
    \item  For the class $H^3_{b}$, the sharp $\mathcal{S}^*(\alpha)$ radius $\rho_3\in(0, 1)$ is the smallest  root  of the equation $x_3(r)=0$, where\\
    $x_3(r)=\alpha-1+\alpha (l+q) r+(7+\alpha+(1+\alpha) l q) r^2+6 (l+q) r^3+(9+3 l q-\alpha (1+l q)) r^4-(-2+\alpha) (l+q) r^5+(1-\alpha) r^6$.
\end{enumerate}
\end{theorem}
\begin{proof} Radius estimate for respective classes is as follows:
\begin{enumerate}[(i)]
    \item
       Note that $x_1(0)=(\alpha-1) <0$ and $x_1(1)=6 (2+d) (2+e) (2+q)>0$ and thus in view of the intermediate value theorem, there exists a root  of the  equation $x_1(r)=0$  in the interval $(0,1)$. Let  $r=\rho_1\in(0, 1)$ be the smallest root of the equation $x_1(r)=0$.
      For $f\in H^1_{b,c}$, using  \eqref{eq0.5}, we have
      \begin{equation}\label{eq3.1}
      \left| \frac{zf'(z)}{f(z)}-1 \right|\leq\frac{r}{(1-r^2)}\left( \frac{dr^2+4r+d}{r^2+dr+1}+\frac{sr^2+4r+s}{r^2+sr+1}+\frac{q r^2+4r+q}{r^2+q r+1}\right)
      \end{equation}
      which yields
      \begin{align}\label{eq0.9}
      \RE\left(\frac{zf'(z)}{f(z)}\right)& \geq1-\frac{r}{(1-r^2)}\left( \frac{dr^2+4r+d}{r^2+dr+1}+\frac{sr^2+4r+s}{r^2+sr+1}+\frac{q r^2+4r+q}{r^2+q r+1}\right)\\
      &\geq\alpha\nonumber
      \end{align}
      whenever $x_1(r)\le0$.
      This shows that  $\RE(zf'(z)/f(z))\geq\alpha$ for $|z|=r\leq\rho_1$.

  For $u=6b-4c\geq0$, $v=4c-q\geq0$, using \eqref{eq0.9},   the function $F_1(z)$ defined for the class $H^1_{b,c}$ in \eqref{eq0.8}  at $z=-\rho_1$ satisfies the equality
     \begin{align*}
    \quad\quad\RE\left(\frac{zF_1'(z)}{F_1(z)}\right)&=1-\frac{\rho_1}{(1-\rho_1^2)}\left( \frac{u\rho_1^2+4\rho_1+u}{\rho_1^2+u\rho_1+1}+\frac{v\rho_1^2+4\rho_1+v}{\rho_1^2+v\rho_1+1}+\frac{q \rho_1^2+4\rho_1+q}{\rho_1^2+q \rho_1+1}\right)\\
    &=\alpha.
     \end{align*}
Thus, the radius is sharp.


 \item  A calculation shows that $x_2(0)=1-\alpha>0$ and $x_2(1/3)=-(-60 (-5+(-12+n) n)+6 (135+n (148+21 n)) q+3 m (9 (30+17 q)+n (296+42 n+156 q+27 n q))+2 (10+3 m) (3+n) (5+3 n) (10+3 q) \alpha )/(2 (10+3 m) (3+n) (5+3 n) (10+3 q))<0$. By  intermediate value theorem, there exists a root $r\in(0, 1/3)$ of the equation $x_2(r) = 0$. Let $\rho_2\in(0, 1/3)$  be the smallest root of the equation $x_2(r)=0$ and $f\in H^2_{b,c}$.
        An easy calculation shows that  for $0<r<1/3$,
      \begin{equation}\label{eqn2}
      C_b-D_b-\sqrt{C_1/2}=\frac{-1+4 r^2+2 b^2 r^2+8 b r^3+r^4+2 b^2 r^4}{2 (-1+r) (1+r) \left(1+2 b r+r^2\right)^2}>0,
      \end{equation}
      where  $C_b$ and $D_b$ are given by  \eqref{eqn1}.
      From \eqref{eq1.01} and \eqref{eq1.1}, and using Lemma \ref{lm2} together with \eqref{eqn2}, we have
     \begin{align}\label{eq2.4}
     &\RE\left(\frac{zf'(z)}{f(z)}\right)\\
     &\geq 1-\left(\frac{r}{1-r^2}\right)\left(\frac{m r^2+4r+m}{r^2+m r+1}+\frac{q r^2+4r+q}{r^2+q r+1}\right)-\frac{r \left(n+2 r+n r^2\right)}{(1+n r) \left(1+2 n r+r^2\right)}\geq\alpha\nonumber
     \end{align}
      whenever $x_2(r)\leq0$. Thus, $\RE(zf'(z)/f(z))\geq\alpha$ for $|z|=r\leq\rho_2$.
     To prove the sharpness, consider the function $F_2, G_2:\Delta\rightarrow \mathbb{C}$ defined by
      \[F_2(z)=\frac{z (1+3 c z-q z) (1-z^2)^2}{(1+5 b z-3 c z+z^2) (1+6 c z-2 q z+z^2) (1+q z+z^2)}\]
      and
      \[G_2(z)=z\frac{1-z^2}{(1+q z+z^2)(1+3c z-q z)},\]
       where $|5b-3c|\leq2$ and $|3c-q|\leq1.$
       Note that for $c=(1+q)/3$,
       \[\frac{F_2(z)}{G_2(z)}=\frac{ (1-z^2)}{1-(1-5 b+q) z+z^2}=\frac{1+w_1(z)}{1-w_1(z)}\]
       and
       \[\frac{G_2(z)}{zp_1(z)}=\frac{1}{1+3 c z-q z}=\frac{1+w_2(z)}{1-w_2(z)},\]
where \[w_1(z)= \frac{(1-5 b+q-2 z) z}{2+(-1+5 b-q) z}\quad\text{and} \quad w_2(z)=\frac{(q-3 c) z}{2+(3 c -q) z}\] which are Schwarz functions and hence $\RE(F_2/G_2)>0$,  $\RE(G_2/(zp_1))>0$ and $\RE(p_1)>0$.
      For $5b-3c\geq0$, $3c-q=1$ and $z=\rho_2$, it follows from \eqref{eq2.4} that
      \begin{align*}
      &\RE\left(\frac{zF_2'(z)}{F_2(z)}\right)\\ &= 1-\left(\frac{\rho_2}{1-\rho_2^2}\right)\left(\frac{m \rho_2^2+4\rho_2+m}{\rho_2^2+m \rho_2+1}+\frac{q \rho_2^2+4\rho_2+q}{\rho_2^2+q \rho_2+1}\right)-\frac{\rho_2 \left(n+2 \rho_2+n \rho_2^2\right)}{(1+n \rho_2) \left(1+2 n \rho_2+\rho_2^2\right)}\\
      &= \alpha.\nonumber
      \end{align*}

    \item   It is easy to see that  $x_3(0)=\alpha-1<0$ and $x_3(1)=4 (2+l) (2+q)>0$. By intermediate value
theorem, there exists a root $r\in(0, 1)$ of the equation $x_3(r) = 0$. Let  $\rho_3\in(0, 1)$ be the smallest root of the equation $x_3(r)=0$. From  \eqref{eq02.2} it follows that for any  $f\in H^3_{b}$,
\begin{align}\label{eq2.6}
\RE\left(\frac{zf'(z)}{f(z)}\right)& \geq1-\frac{r}{1-r^2} \left(\frac{l r^2+4 r+l}{r^2+l r+1}+\frac{q r^2+4 r+q}{r^2+q r +1}\right)\geq\alpha
\end{align}
 whenever $x_3(r)\leq0$. This proves that  $\RE(zf'(z)/f(z))\geq\alpha$  for $|z|=r\leq\rho_3$.
 The result is sharp for the function $F_3$ defined for the class $H^3_{b}$ in \eqref{eq2.5}.  At $z=-\rho_3$, and for $u=4b-q\geq0$,
 it follows from \eqref{eq2.6} that
\begin{align*}
\RE\left(\frac{zF_3'(z)}{F_3(z)}\right)& =1-\frac{\rho_3}{1-\rho_3^2} \left(\frac{u \rho_3^2+4 \rho_3+u}{\rho_3^2+u \rho_3+1}+\frac{q \rho_3^2+4 \rho_3+q}{\rho_3^2+q \rho_3 +1}\right)=\alpha.
\end{align*}
\end{enumerate}
\noindent
\end{proof}
\noindent

\begin{remark}
For $b=1$ and $c=1$ and $q=2$,  Theorem \ref{th1}  yields the result \cite[Theorem 1, p. 6]{lecko}.
\end{remark}

The class $\mathcal{S}_L^*$=$\mathcal{S}^*(\sqrt{1+z})$ and it represents the collection of functions in the class $\A$ whose $zf'(z)/f(z)$ lies in the region bounded by the lemniscate of Bernoulli $|w^2-1|=1.$ Various studies on $\mathcal{S}_L^*$ can be seen in \cite{sokol, mad}. In the following result  we obtain the sharp   radii results for the classes $H^1_{b,c}$, $H^2_{b,c}$ and $H^3_b$.

\begin{theorem}\label{th2}
The sharp $\mathcal{S}_L^*$ radii for the classes $H^1_{b,c}$,  $H^2_{b,c}$ and $H^3_{b}$ are as follows:
\begin{enumerate}[(i)]
    \item For the class $H^1_{b,c}$, the sharp $\mathcal{S}_L^*$ radius $\rho_1\in(0, 1)$ is the smallest  root  of the equation $x_1(r)=0$, where\\
    $x_1(r)=(1-\sqrt{2})-(-2+\sqrt{2}) (d+s+q) r+(-2 (-7+\sqrt{2})-(-3+\sqrt{2}) (s q+d (s+q))) r^2+(-(-12+\sqrt{2}) (s+q)-d (-12+\sqrt{2}+(-4+\sqrt{2}) s q)) r^3+8 (3+s q+d (s+q)) r^4+((10+\sqrt{2}) (s+q)+d (10+\sqrt{2}+(2+\sqrt{2}) s q)) r^5+(2 (5+\sqrt{2})+(1+\sqrt{2}) (s q+d (s+q))) r^6+\sqrt{2} (d+s+q) r^7+(-1+\sqrt{2}) r^8$.
    \item For the class $H^2_{b,c}$, the sharp $\mathcal{S}_L^*$ radius $\rho_2\in(0, 1)$ is the smallest  root  of the equation $x_2(r)=0$, where\\
    $ x_2(r)=(1-\sqrt{2})-(-2+\sqrt{2}) (m+n+q) r+(11-\sqrt{2}-(-3+\sqrt{2}) (n q+m (n+q))) r^2+(8 (m+q)-n (-12+\sqrt{2}+(-4+\sqrt{2}) m q)) r^3+(11+\sqrt{2}+(3+\sqrt{2}) m q+8 n (m+q)) r^4+((10+\sqrt{2}) n+(2+\sqrt{2}) q+(2+\sqrt{2}) m (1+n q)) r^5+(1+\sqrt{2}) (1+n (m+q)) r^6+\sqrt{2} n r^7$.
    \item For the class $H^3_{b}$, the sharp $\mathcal{S}_L^*$ radius $\rho_3\in(0, 1)$ is the smallest  root  of the equation $x_3(r)=0$, where\\
    $x_3(r)=(1-\sqrt{2})-(-2+\sqrt{2}) (l+q) r+(9-\sqrt{2}-(-3+\sqrt{2}) l q) r^2+6 (l+q) r^3+(7+\sqrt{2}+(1+\sqrt{2}) l q) r^4+\sqrt{2} (l+q) r^5+(-1+\sqrt{2}) r^6$.
\end{enumerate}
\end{theorem}
\begin{proof}
\begin{enumerate}[(i)]
\item Note that $x_1(0)=(1-\sqrt{2}) <0$ and $x_1(1)=6 (2 + d) (2 + q) (2 + s)>0$ and thus in view of the intermediate value theorem, there exists a root  of the  equation $x_1(r)=0$  in the interval $(0,1)$. Let  $r=\rho_1\in(0, 1)$ be the smallest root of the equation $x_1(r)=0$.
     Ali \emph{et al.} \cite[Lemma 2.2]{ali} had proved, for $2\sqrt{2}/3<C<\sqrt{2},$ that
     \begin{equation}\label{eq3.2}
     \left\{w \in \C: |w-C|< \sqrt{2}-C \right\} \subset \left\{w \in \C: |w^2-1|< 1 \right\}.
     \end{equation}
     In view of  \eqref{eq3.2} and the fact that the centre of the disc in \eqref{eq3.1} is 1,   $f\in \mathcal{S}^*_L$ if
       \begin{equation}\label{eq3.3}
       \frac{r}{(1-r^2)}\left( \frac{dr^2+4r+d}{r^2+dr+1}+\frac{sr^2+4r+s}{r^2+sr+1}+\frac{q r^2+4r+q}{r^2+q r+1}\right)\leq\sqrt{2}-1
       \end{equation}
       which is equivalent to $f\in \mathcal{S}^*_L$ if  $x_1(r)\leq0$. Since $x_1(0)<0$ and $\rho_1$ is the smallest root of the equation $x_1(r)=0$,  $x_1(r)$ is an increasing function on $(0, \rho_1)$. In view of this  $f\in \mathcal{S}^*_L$ for $|z|=r\leq\rho_1$.
       For $u=6b-4c\geq0$, $v=4c-q\geq0$, using \eqref{eq3.3},   the function $f_1(z)$ defined for the class $H^1_{b,c}$ in \eqref{eq0.6}  at $z=-\rho_1$, satisfies the  following equality
     \begin{align*}
    &\left|\left(\frac{zf_1'(z)}{f_1(z)}\right)^2-1\right|\\
    &=\left|\left(1-\frac{\rho_1}{(1-\rho_1^2)}\left( \frac{u\rho_1^2+4\rho_1+u}{\rho_1^2+u\rho_1+1}+\frac{v\rho_1^2+4\rho_1+v}{\rho_1^2+v\rho_1+1}+\frac{q \rho_1^2+4\rho_1+q}{\rho_1^2+q \rho_1+1}\right)\right)^2-1\right|=1.
         \end{align*}
 Thus, the radius is sharp.


 \item  A calculation shows that $x_2(0)=1-\sqrt2<0$ and $x_2(1)= 6(2+m)(1+n)(2+q)>0$. By  intermediate value theorem, there exists a root $r\in(0, 1)$ of the equation $x_2(r) = 0$. Let $\rho_2\in(0, 1)$  be the smallest root of the equation $x_2(r)=0$ and $f\in H^2_{b,c}$.
     As  the centre of the disc in \eqref{eq2.3} is 1,  by \eqref{eq3.2}, $f\in \mathcal{S}^*_L$ if
       \begin{equation}\label{eq3.4}
       \frac{r}{(1-r^2)}\left( \frac{mr^2+4r+m}{r^2+mr+1}+\frac{nr^2+2r+n}{nr+1}+\frac{q r^2+4r+q}{r^2+q r+1}\right)\leq\sqrt2-1
       \end{equation}
    which is equivalent to $f\in \mathcal{S}^*_L$ if  $x_2(r)\leq0$. Since $x_2(0)<0$ and $\rho_2$ is the smallest root of the equation $x_2(r)=0$,  $x_2(r)$ is an increasing function on $(0, \rho_2)$. Thus, $f\in \mathcal{S}^*_L$ for $|z|=r\leq\rho_2$. To prove the sharpness, consider the function $f_2$ defined in \eqref{eq1.3}. For $u=5b-3c\geq0$, $v=3c-q\geq0$ and $z=-\rho_2$, it follows from \eqref{eq3.4} that
      \begin{align*}
      &\left|\left(\frac{zf_2'(z)}{f_2(z)}\right)^2-1\right|\\
      & =\left|\left(1-\frac{\rho_2}{(1-\rho_2^2)}\left( \frac{u\rho_2^2+4\rho_2+u}{\rho_2^2+u\rho_2+1}+\frac{v\rho_2^2+2\rho_2+v}{v\rho_2+1}+\frac{q \rho_2^2+4\rho_2+q}{\rho_2^2+q \rho_2+1}\right)\right)^2-1\right|=1
      \end{align*}
\item It is easy to see that $x_3(0)=1-\sqrt2<0$ and $x_3(1)=4 (2+l) (2+q)>0$, by intermediate  value
theorem, there exists a root $r\in(0, 1)$ of the equation $x_3(r) = 0$. Let  $\rho_3\in(0, 1)$ be the smallest root of the equation $x_3(r)=0$. From \eqref{eq2.1} and \eqref{eq2.2} it follows that for any  $f\in H^3_{b}$
\begin{align}\label{eq3.5}
\left|\frac{zf'(z)}{f(z)}-1\right|\leq \frac{r}{1-r^2} \left(\frac{l r^2+4 r+l}{r^2+l r+1}+\frac{q r^2+4 r+q}{r^2+q r +1}\right)
\end{align}

As  the centre of the disc in \eqref{eq3.5} is 1,  by \eqref{eq3.2}, $f\in \mathcal{S}^*_L$ if
       \begin{equation}\label{eq3.67}
       \frac{r}{1-r^2} \left(\frac{l r^2+4 r+l}{r^2+l r+1}+\frac{q r^2+4 r+q}{r^2+q r +1}\right)\leq\sqrt2-1
       \end{equation}
    which is equivalent to $f\in \mathcal{S}^*_L$ if  $x_3(r)\leq0$. Since $x_3(0)<0$ and $\rho_3$ is the smallest root of the equation $x_3(r)=0$,  $x_3(r)$ is an increasing function on $(0, \rho_3)$. This proves that $f\in \mathcal{S}^*_L$ for $|z|=r\leq\rho_3$.

        \noindent
 The result is sharp for the function $f_3$ defined for the class $H^3_{b}$ in \eqref{eq0.10}.  At $z=-\rho_3$, and for $u=4b-q\geq0$,
 it follows from \eqref{eq3.67} that
\begin{align*}
      &\left|\left(\frac{zf_3'(z)}{f_3(z)}\right)^2-1\right| = \left|\left(\frac{\rho_3}{1-\rho_3^2} \left(\frac{u \rho_3^2+4 \rho_3+u}{\rho_3^2+u \rho_3+1}+\frac{q \rho_3^2+4 \rho_3+q}{\rho_3^2+q \rho_3 +1}\right)\right)^2-1\right|=1.
\end{align*}
\end{enumerate}
\noindent
\end{proof}
\begin{remark}
For $b=1$ and $c=1$ and $q=2$,  Theorem \ref{th2}  yields the result \cite[Theorem 2, p. 8]{lecko}.
\end{remark}
\noindent
For
$\phi_{PAR}= 1+\frac{2}{{\pi}^2}\left(\log \left(\dfrac{1+\sqrt{z}}{1-\sqrt{z}}\right)\right)^2$
  the class $\mathcal{S}^*_p:= \mathcal{S}^*(\phi_{PAR})$ is the class of parabolic starlike functions. A function $f \in \mathcal{S}^*_p$ provided $zf'(z)/f(z)$ lies in the parabolic region given by $\{w\in \C:\RE(w) >|w-1|\}$. For further reading refer to \cite{AV,grs,MM1,ravi}.
The following theorem gives the radius of parabolic starlikeness of the three classes $H^1_{b,c}$, $H^2_{b,c}$ and $H^3_b$.
\begin{theorem}\label{th3}
The $\mathcal{S}_p^*$ radii for the classes $H^1_{b,c}$,  $H^2_{b,c}$ and $H^3_{b}$ are as follows:
\begin{enumerate}[(i)]
    \item For the class $H^1_{b,c}$, the sharp $\mathcal{S}_p^*$ radius $\rho_1\in(0, 1)$ is the smallest  root  of the equation $x_1(r)=0$, where\\
    $x_1(r)=-1+(d+s+q) r+(22+3 s q+3 d (s+q)) r^2+(21 (s+q)+d (21+5 s q)) r^3+16 (3+s q+d (s+q)) r^4+(23 (s+q)+d (23+7 s q)) r^5+(26+5 s q+5 d (s+q)) r^6+3 (d+s+q) r^7+r^8.$
    \item For the class $H^2_{b,c}$,   $\mathcal{S}_p^*$ radius $\rho_2\in(0, 1)$ is the smallest  root  of the equation $x_2(r)=0$, where\\
    $x_2(r)=-1+(m+q) r+(18+4 m n+3 m q+4 n q) r^2+(17 m+36 n+17 q+8 m n q) r^3+(40+28 m n+12 m q+28 n q) r^4+(19 m+40 n+19 q+12 m n q) r^5+(22+8 m n+5 m q+8 n q) r^6+(3 m+4 n+3 q) r^7+r^8.$
    \item For the class $H^3_{b}$, the sharp $\mathcal{S}_p^*$ radius $\rho_3\in(0, 1)$ is the smallest  root  of the equation $x_3(r)=0$, where\\
    $x_3(r)=-1+(l+q) r+3 (5+l q) r^2+12 (l+q) r^3+(17+5 l q) r^4+3 (l+q) r^5+r^6.$
\end{enumerate}
\end{theorem}
\begin{proof}
\begin{enumerate}[(i)]
    \item Note that $x_1(0)=-1 <0$ and $x_1(1)=12 (2 + d) (2 + q) (2 + s)>0$ and thus in view of the intermediate value theorem, there exists a root  of the  equation $x_1(r)=0$  in the interval $(0,1)$. Let  $r=\rho_1\in(0, 1)$ be the smallest root of the equation $x_1(r)=0$.
      Shanmughan and Ravichandran \cite[p. 321]{TN} had proved, for $1/2<C<3/2,$ that
      \begin{equation}\label{eq4.1}
      \{w\in\C: |w-C|<C - 1/2\} \subset \{w\in\C: \RE(w) > |w-1|\}.
      \end{equation}
      As  the centre of the disc in \eqref{eq3.1} is 1,  by \eqref{eq4.1}, $f\in \mathcal{S}^*_p$ if
       \begin{equation}\label{eq4.2}
       \frac{r}{(1-r^2)}\left( \frac{dr^2+4r+d}{r^2+dr+1}+\frac{sr^2+4r+s}{r^2+sr+1}+\frac{q r^2+4r+q}{r^2+q r+1}\right)\leq\frac{1}{2}
       \end{equation}
       which is equivalent to $f\in \mathcal{S}^*_p$ if  $x_1(r)\leq0$. Since $x_1(0)<0$ and $\rho_1$ is the smallest root of the equation $x_1(r)=0$,  $x_1(r)$ is an increasing function on $(0, \rho_1)$. In view of this  $f\in \mathcal{S}^*_p$ for $|z|=r\leq\rho_1$.
       For $u=6b-4c\geq0$, $v=4c-q\geq0$, using \eqref{eq4.2},   the function $F_1(z)$ defined for the class $H^1_{b,c}$ in \eqref{eq0.8}  at $z=-\rho_1$, satisfies the  following equality
     \begin{align*}
    \quad\quad\RE\left(\frac{zF_1'(z)}{F_1(z)}\right)&=\left|1-\frac{\rho_1}{(1-\rho_1^2)}\left( \frac{u\rho_1^2+4\rho_1+u}{\rho_1^2+u\rho_1+1}+\frac{v\rho_1^2+4\rho_1+v}{\rho_1^2+v\rho_1+1}+\frac{q \rho_1^2+4\rho_1+q}{\rho_1^2+q \rho_1+1}\right)\right|\\
    &=\left|\frac{zF_1'(z)}{F_1(z)}-1\right|.
         \end{align*}
 Thus, the radius is sharp.
%
    \item   A calculation shows that $x_2(0)=-1<0$ and \[x_2\left(\frac{1}{3}\right)= \frac{4 (9 m (190+89 q+n (146+63 q))+2 (1250+855 q+3 n (410+219 q)))}{6561}\] which is greater than 0. By  intermediate value theorem, there exists a root $r\in(0, 1/3)$ of the equation $x_2(r)$. Let $\rho_2\in(0, 1/3)$  be the smallest root of the equation $x_2(r)=0$ and $f\in H^2_{b,c}$.
        From \eqref{eq1.01} and \eqref{eq1.1}, and using Lemma \ref{lm2} together with \eqref{eqn2}, we have
     \begin{align*}
     &\RE\left(\frac{zf'(z)}{f(z)}\right)\\
     &\geq 1-\left(\frac{r}{1-r^2}\right)\left(\frac{m r^2+4r+m}{r^2+m r+1}+\frac{q r^2+4r+q}{r^2+q r+1}\right)-\frac{r \left(n+2 r+n r^2\right)}{(1+n r) \left(1+2 n r+r^2\right)}\\
     &\geq\frac{r}{(1-r^2)}\left( \frac{mr^2+4r+m}{r^2+mr+1}+\frac{nr^2+2r+n}{nr+1}+\frac{q r^2+4r+q}{r^2+q r+1}\right)\geq
     \left|\frac{zf_2'(z)}{f_2(z)}-1\right|\nonumber
     \end{align*}
      whenever $x_2(r)\leq0$.
      Since $x_2(0)<0$ and $\rho_2$ is the smallest root of the equation $x_2(r)=0$,  $x_2(r)$ is an increasing function on $(0, \rho_2)$. Thus, $f\in \mathcal{S}^*_p$ for $|z|=r\leq\rho_2$.

    \item It is easy to see that $x_3(0)=1<0$ and $x_3(1)=8 (2+l) (2+q)>0$. By intermediate value
theorem, there exists a root $r\in(0, 1)$ of the equation $x_3(r) = 0$. Let  $\rho_3\in(0, 1)$ be the smallest root of the equation $x_3(r)=0$.
In view of  \eqref{eq4.1} and the fact that  the centre of the disc in \eqref{eq3.5} is 1, $f\in \mathcal{S}^*_p$ if
       \begin{equation}\label{eq3.6}
       \frac{r}{1-r^2} \left(\frac{l r^2+4 r+l}{r^2+l r+1}+\frac{q r^2+4 r+q}{r^2+q r +1}\right)\leq\frac{1}{2}
       \end{equation}
    which is equivalent to $f\in \mathcal{S}^*_p$ if  $x_3(r)\leq0$. Since $x_3(0)<0$ and $\rho_3$ is the smallest root of the equation $x_3(r)=0$,  $x_3(r)$ is an increasing function on $(0, \rho_3)$. This proves that $f\in \mathcal{S}^*_p$ for $|z|=r\leq\rho_3$.

\noindent
 The result is sharp for the function $F_3$ defined for the class $H^3_{b}$ in \eqref{eq2.5}.  At $z=-\rho_3$, and for $u=4b-q\geq0$,
 it follows from \eqref{eq3.6} that
\[\quad\RE\left(\frac{zF_3'(z)}{F_3(z)}\right) = \left|1-\frac{\rho_3}{1-\rho_3^2} \left(\frac{u \rho_3^2+4 \rho_3+u}{\rho_3^2+u \rho_3+1}+\frac{q \rho_3^2+4 \rho_3+q}{\rho_3^2+q \rho_3 +1}\right)\right|=\left|\frac{zF_3'(z)}{F_3(z)}-1\right|.\qedhere\]
\end{enumerate}
\end{proof}
\begin{remark}
 Putting $b=1$ and $c=1$ and $q=2$ in  Theorem \ref{th3}, we get the result \cite[Theorem 3, p. 9]{lecko} with the part$(ii)$ having improved radius ($=0.0990195>0.0972$).
\end{remark}
\noindent
In 2015,  the class of starlike functions associated with the exponential function as $\mathcal{S}_e^*=\mathcal{S}^*(e^{z})$ was introduced by  Mendiratta \emph{et al.}\cite{mnr}. It satisfies the condition $|\log (zf'(z)/f(z))|<1$.
\begin{theorem}\label{th4}
The  $\mathcal{S}_e^*$ radii for the classes $H^1_{b,c}$,  $H^2_{b,c}$ and $H^3_{b}$ are as follows:
\begin{enumerate}[(i)]
    \item  For the class $H^1_{b,c}$, the sharp $\mathcal{S}_e^*$ radius $\rho_1\in(0, 1)$ is the smallest  root  of the equation $x_1(r)=0$, where\\
    $ x_1(r)=(1-e)+(d+q+s) r+(2+10 e+d q+d e q+(1+e) (d+q) s) r^2+(q+s+10 e (q+s)+d (1+q s+2 e (5+q s))) r^3+8 e (3+q s+d (q+s)) r^4+((-1+12 e) (q+s)+d (-1-q s+4 e (3+q s))) r^5+(-2+14 e-d q+3 d e q+(-1+3 e) (d+q) s) r^6+(-1+2 e) (d+q+s) r^7+(-1+e) r^8$.
    \item  For the class $H^2_{b,c}$,  $\mathcal{S}_e^*$ radius $\rho_2\in(0, 1)$ is the smallest  root  of the equation $x_2(r)=0$, where\\
    $x_2(r)=(1-e)+(m+n+q) r+(1+9 e+(1+e) (n q+m (n+q))) r^2+(8 e (m+q)+n (1+m q+2 e (5+m q))) r^3+(-1-m q+e (13+8 m n+5 m q+8 n q)) r^4+(-n-q+4 e (3 n+q)+m (-1+4 e) (1+n q)) r^5+(-1+3 e) (1+n (m+q)) r^6+n (-1+2 e)r^7.$
    \item  For the class $H^3_{b}$, the sharp $\mathcal{S}_e^*$ radius $\rho_3\in(0, 1)$ is the smallest  root  of the equation $x_3(r)=0$, where\\
    $x_3(r)=(1-e)+(l+q) r+(1+7 e+(1+e) l q) r^2+6 e (l+q) r^3+(-1+9 e+(-1+3 e) l q) r^4+(-1+2 e) (l+q) r^5+(-1+e) r^6 $.
\end{enumerate}
\end{theorem}
\begin{proof}
\begin{enumerate}[(i)]
    \item Note that $x_1(0)=1-e <0$ and $x_1(1)=6e(2 + d) (2 + q) (2 + s)>0$ and thus in view of the intermediate value theorem, there exists a root  of the  equation $x_1(r)=0$  in the interval $(0,1)$. Let  $r=\rho_1\in(0, 1)$ be the smallest root of the equation $x_1(r)=0$.
      Mendiratta \emph{et al.}\cite{mnr} proved, for $e^{-1} \le C \le (e+e^{-1})/2)$, that
      \begin{equation}\label{eq5.1}
      \{w \in \mathbb{C}: |w-C|< C-e^{-1}\} \subset \{w \in \mathbb{C}: |\log (w)|<1\}.
      \end{equation}
      As  the centre of the disc in \eqref{eq3.1} is 1,  by \eqref{eq5.1}, $f\in \mathcal{S}^*_e$ if
       \begin{equation}\label{eq5.2}
       \frac{r}{(1-r^2)}\left( \frac{dr^2+4r+d}{r^2+dr+1}+\frac{sr^2+4r+s}{r^2+sr+1}+\frac{q r^2+4r+q}{r^2+q r+1}\right)\leq1-\frac{1}{e}
       \end{equation}
       which is equivalent to $f\in \mathcal{S}^*_e$ if  $x_1(r)\leq0$. Since $x_1(0)<0$ and $\rho_1$ is the smallest root of the equation $x_1(r)=0$,  $x_1(r)$ is an increasing function on $(0, \rho_1)$. In view of this  $f\in \mathcal{S}^*_e$ for $|z|=r\leq\rho_1$.
       For $u=6b-4c\geq0$, $v=4c-q\geq0$, using \eqref{eq5.2},   the function $F_1(z)$ defined for the class $H^1_{b,c}$ in \eqref{eq0.8}  at $z=-\rho_1$, satisfies the  following equality
     \begin{align*}
    &\left|\log\left(\frac{zF_1'(z)}{F_1(z)}\right)\right|\\
    &=\left|\log\left(1-\frac{\rho_1}{(1-\rho_1^2)}\left( \frac{u\rho_1^2+4\rho_1+u}{\rho_1^2+u\rho_1+1}+\frac{v\rho_1^2+4\rho_1+v}{\rho_1^2+v\rho_1+1}+\frac{q \rho_1^2+4\rho_1+q}{\rho_1^2+q \rho_1+1}\right)\right)\right|\\
    &=\left|\log\left(\frac{1}{e}\right)\right|=1.
         \end{align*}
 Thus, the radius is sharp.
\item  A calculation shows that $x_2(0)=1-e<0$ and $x_2(1)= 6e(2+m)(1+n)(2+q)>0$. By  intermediate value theorem, there exists a root $r\in(0, 1)$ of the equation $x_2(r) = 0$. Let $\rho_2\in(0, 1)$  be the smallest root of the equation $x_2(r)=0$ and $f\in H^2_{b,c}$. In view of  \eqref{eq5.1} and the fact that the centre of the disc in \eqref{eq2.3} is 1,  $f\in \mathcal{S}^*_e$ if
       \begin{equation*}
       \frac{r}{(1-r^2)}\left( \frac{mr^2+4r+m}{r^2+mr+1}+\frac{nr^2+2r+n}{nr+1}+\frac{q r^2+4r+q}{r^2+q r+1}\right)\leq1-\frac{1}{e}
       \end{equation*}
    which is equivalent to $f\in \mathcal{S}^*_e$ if  $x_2(r)\leq0$. Since $x_2(0)<0$ and $\rho_2$ is the smallest root of the equation $x_2(r)=0$,  $x_2(r)$ is an increasing function on $(0, \rho_2)$. Thus, $f\in \mathcal{S}^*_e$ for $|z|=r\leq\rho_2$.
    \item It is easy to see that $x_3(0)=1-e<0$ and $x_3(1)=4e(2+l)(2+q)>0$. By intermediate  value
theorem, there exists a root $r\in(0, 1)$ of the equation $x_3(r) = 0$. Let  $\rho_3\in(0, 1)$ be the smallest root of the equation $x_3(r)=0$.
Since  the centre of the disc in \eqref{eq3.5} is 1,  by \eqref{eq5.1}, $f\in \mathcal{S}^*_e$ if
       \begin{equation}\label{eq5.6}
       \frac{r}{1-r^2} \left(\frac{l r^2+4 r+l}{r^2+l r+1}+\frac{q r^2+4 r+q}{r^2+q r +1}\right)\leq1-\frac{1}{e}
       \end{equation}
    which is equivalent to $f\in \mathcal{S}^*_e$ if  $x_3(r)\leq0$. Since $x_3(0)<0$ and $\rho_3$ is the smallest root of the equation $x_3(r)=0$,  $x_3(r)$ is an increasing function on $(0, \rho_3)$. This proves that $f\in \mathcal{S}^*_e$ for $|z|=r\leq\rho_3$.

 The result is sharp for the function $F_3$ defined for the class $H^3_{b}$ in \eqref{eq2.5}.  At $z=-\rho_3$, and for $u=4b-q\geq0$,
 it follows from \eqref{eq5.6} that
\begin{align*}
\left|\log\left(\frac{zF_3'(z)}{F_3(z)}\right)\right|& = \left|\log\left(1-\frac{\rho_3}{1-\rho_3^2} \left(\frac{u \rho_3^2+4 \rho_3+u}{\rho_3^2+u \rho_3+1}+\frac{q \rho_3^2+4 \rho_3+q}{\rho_3^2+q \rho_3 +1}\right)\right)\right|\\
&=\left|\log\left(\frac{1}{e}\right)\right|=1.\qedhere
\end{align*}
 \end{enumerate}
\end{proof}
\begin{remark}
For $b=1$ and $c=1$ and $q=2$,  Theorem \ref{th4}  reduces to the  result \cite[Theorem 4, p. 10]{lecko}.
\end{remark}

\noindent
The class $\mathcal{S}_c^*=\mathcal{S}^*(1+(4/3)z+(2/3)z^2)$ is the class of starlike functions $f \in \mathcal{A}$ such that  $zf'(z)/f(z)$ lies in the region  bounded by the cardiod $\Omega_c=\{u+iv: (9u^2+9v^2-18u+5)^2-16(9u^2+ 9 v^2-6u+1)=0\}.$  Sharma \emph{et al.} \cite{sharma} studied various properties of the class $\mathcal{S}_c^*$. The following theorem determines radii  for starlike functions associated with the cardiod
\begin{theorem}\label{th5}
The  $\mathcal{S}_c^*$ radii for the classes $H^1_{b,c}$,  $H^2_{b,c}$ and $H^3_{b}$ are as follows:
\begin{enumerate}[(i)]
    \item For the class $H^1_{b,c}$, the sharp $\mathcal{S}_c^*$ radius $\rho_1\in(0, 1)$ is the smallest  root  of the equation $x_1(r)=0$, where\\
    $x_1(r)=-2+(d+q+s) r+4 (8+q s+d (q+s)) r^2+(31 (q+s)+d (31+7 q s)) r^3+24 (3+q s+d (q+s)) r^4+(35 (q+s)+d (35+11 q s)) r^5+8 (5+q s+d (q+s)) r^6+5 (d+q+s) r^7+2 r^8$.
    \item For the class $H^2_{b,c}$, the  $\mathcal{S}_c^*$ radius $\rho_2\in(0, 1)$ is the smallest  root  of the equation $x_2(r)=0$, where\\
    $x_2(r)=-2+(m+n+q) r+4 (7+n q+m (n+q)) r^2+(31 n+24 q+m (24+7 n q)) r^3+(38+24 m n+24n q+14 m q) r^4+(35 n+11 q+11 m (1+n q)) r^5+8 (1+n (m+q)) r^6+5 n r^7$.
    \item For the class $H^3_{b}$, the sharp $\mathcal{S}_c^*$ radius $\rho_3\in(0, 1)$ is the smallest  root  of the equation $x_3(r)=0$, where\\
    $x_3(r)=-2+(l+q) r+(22+4 l q) r^2+18 (l+q) r^3+(26+8 l q) r^4+5 (l+q) r^5+2 r^6$.
\end{enumerate}
\end{theorem}
\begin{proof}
\begin{enumerate}[(i)]
    \item
Note that $x_1(0)=-2 <0$ and $x_1(1)=18(2 + d) (2 + q) (2 + s)>0$ and thus in view of the intermediate value theorem, there exists a root  of the  equation $x_1(r)=0$  in the interval $(0,1)$. Let  $r=\rho_1\in(0, 1)$ be the smallest root of the equation $x_1(r)=0$.
      Sharma \emph{et al.} \cite{sharma} proved that for $1/3<C \le 5/3$,
      \begin{equation}\label{eq6.1}
      \{ w \in \mathbb{C}: |w-C|<(3C-1)/3\} \subseteq \Omega_c.
      \end{equation}
      As  the centre of the disc in \eqref{eq3.1} is 1,  by \eqref{eq6.1}, $f\in \mathcal{S}^*_c$ if
       \begin{equation}\label{eq6.2}
       \frac{r}{(1-r^2)}\left( \frac{dr^2+4r+d}{r^2+dr+1}+\frac{sr^2+4r+s}{r^2+sr+1}+\frac{q r^2+4r+q}{r^2+q r+1}\right)\leq\frac{2}{3}
       \end{equation}
       which is equivalent to $f\in \mathcal{S}^*_c$ if  $x_1(r)\leq0$. Since $x_1(0)<0$ and $\rho_1$ is the smallest root of the equation $x_1(r)=0$,  $x_1(r)$ is an increasing function on $(0, \rho_1)$. In view of this  $f\in \mathcal{S}^*_c$ for $|z|=r\leq\rho_1$.
       For $u=6b-4c\geq0$, $v=4c-q\geq0$, using \eqref{eq6.2},   the function $F_1(z)$ defined for the class $H^1_{b,c}$ in \eqref{eq0.8}  at $z=-\rho_1$, satisfies the  following equality
     \begin{align*}
    \left|\frac{zF_1'(z)}{F_1(z)}\right|=\left|1-\frac{\rho_1}{(1-\rho_1^2)}\left( \frac{u\rho_1^2+4\rho_1+u}{\rho_1^2+u\rho_1+1}+\frac{v\rho_1^2+4\rho_1+v}{\rho_1^2+v\rho_1+1}+\frac{q \rho_1^2+4\rho_1+q}{\rho_1^2+q \rho_1+1}\right)\right|=\frac{1}{3},
             \end{align*}
 which belongs to boundary of the region $\Omega_c$. Thus, the radius is sharp.

     \item  A calculation shows that $x_2(0)=-2<0$ and $x_2(1)= 18(2+m)(1+n)(2+q)>0$. By  intermediate value theorem, there exists a root $r\in(0, 1)$ of the equation $x_2(r) = 0$. Let $\rho_2\in(0, 1)$  be the smallest root of the equation $x_2(r)=0$ and $f\in H^2_{b,c}$. In view of  \eqref{eq6.1} and the fact that   the centre of the disc in \eqref{eq2.3} is 1,   $f\in \mathcal{S}^*_c$ if
       \begin{equation*}
       \frac{r}{(1-r^2)}\left( \frac{mr^2+4r+m}{r^2+mr+1}+\frac{nr^2+2r+n}{nr+1}+\frac{q r^2+4r+q}{r^2+q r+1}\right)\leq\frac{2}{3}
       \end{equation*}
    which is equivalent to $f\in \mathcal{S}^*_c$ if  $x_2(r)\leq0$. Since $x_2(0)<0$ and $\rho_2$ is the smallest root of the equation $x_2(r)=0$,  $x_2(r)$ is an increasing function on $(0, \rho_2)$. Thus, $f\in \mathcal{S}^*_c$ for $|z|=r\leq\rho_2$.
\item It is easy to see that $x_3(0)=-2<0$ and $x_3(1)=12(2+l)(2+q)>0$. By intermediate value
theorem, there exists a root $r\in(0, 1)$ of the equation $x_3(r) = 0$. Let  $\rho_3\in(0, 1)$ be the smallest root of the equation $x_3(r)=0$.
In view of   the fact that the  centre of the disc in \eqref{eq3.5} is 1,  by \eqref{eq6.1}, $f\in \mathcal{S}^*_c$ if
       \begin{equation}\label{eq6.4}
       \frac{r}{1-r^2} \left(\frac{l r^2+4 r+l}{r^2+l r+1}+\frac{q r^2+4 r+q}{r^2+q r +1}\right)\leq\frac{2}{3}
       \end{equation}
    which is equivalent to $f\in \mathcal{S}^*_e$ if  $x_3(r)\leq0$. Since $x_3(0)<0$ and $\rho_3$ is the smallest root of the equation $x_3(r)=0$,  $x_3(r)$ is an increasing function on $(0, \rho_3)$. This proves that $f\in \mathcal{S}^*_c$ for $|z|=r\leq\rho_3$.

 The result is sharp for the function $F_3$ defined for the class $H^3_{b}$ in \eqref{eq2.5}.  At $z=-\rho_3$, and for $u=4b-q\geq0$,
 it follows from \eqref{eq6.4} that
\begin{align*}
\left|\frac{zF_3'(z)}{F_3(z)}\right|& = \left|1-\frac{\rho_3}{1-\rho_3^2} \left(\frac{u \rho_3^2+4 \rho_3+u}{\rho_3^2+u \rho_3+1}+\frac{q \rho_3^2+4 \rho_3+q}{\rho_3^2+q \rho_3 +1}\right)\right|=\frac{1}{3}. \qedhere
\end{align*}
 \end{enumerate}
\end{proof}
\begin{remark}
 Putting $b=1$ and $c=1$ and $q=2$ in  Theorem \ref{th5}, we get the result \cite[Theorem 5, p. 11]{lecko}.
\end{remark}
\noindent
In 2019, Cho \emph{et al.} \cite{cho} considered the class of starlike functions $\mathcal{S}_{\sin}^*=\{f \in \A: zf'(z)/f(z) \prec 1+\sin(z):= q_0(z)\}$ associated with the sine function.
\begin{theorem}\label{th6}
The sharp $\mathcal{S}_{\sin}^*$ radii for the classes $H^1_{b,c}$,  $H^2_{b,c}$ and $H^3_{b}$ are as follows:
\begin{enumerate}[(i)]
    \item For the class $H^1_{b,c}$, the sharp $\mathcal{S}_{\sin}^*$ radius $\rho_1\in(0, 1)$ is the smallest  root  of the equation $x_1(r)=0$, where\\
    $x_1(r)=-\sin(1)-(d+q+s) (-1+\sin(1)) r-2 (-6+\sin(1)-(q s+d (q+s)) (\sin(1)-2)) r^2+(11 (q+s)+d (11+3 q s)-(d+q+s+d q s) \sin(1)) r^3+8 (3+q s+d (q+s)) r^4+((q+s) (11+\sin(1))+d (11+\sin(1)+q s (3+\sin(1)))) r^5+((q s+d (q+s)) (2+\sin(1))+2 (6+\sin(1))) r^6+(d+q+s) (1+\sin(1)) r^7+\sin(1) r^8$.
    \item For the class $H^2_{b,c}$, the sharp $\mathcal{S}_{\sin}^*$ radius $\rho_2\in(0, 1)$ is the smallest  root  of the equation $x_2(r)=0$, where\\
    $x_2(r)=-\sin(1)-(m+n+q) (\sin(1)-1) r+(10-(n q+m (n+q)) (-2+\sin(1))-\sin(1)) r^2+(8 m+11 n+8 q+3 m n q-n (1+m q) \sin(1)) r^3+(12+8 n q+\sin(1)+m (8 n+q (4+\sin(1)))) r^4+(q (3+\sin(1))+m (1+n q) (3+\sin(1))+n (11+\sin(1))) r^5+(1+n (m+q)) (2+\sin(1)) r^6+n (1+\sin(1)) r^7$.
    \item For the class $H^3_{b}$, the sharp $\mathcal{S}_{\sin}^*$ radius $\rho_3\in(0, 1)$ is the smallest  root  of the equation $x_3(r)=0$, where\\
     $x_3(r)= -\sin(1)-(l+q) (\sin(1)-1) r+(8-l q (-2+\sin(1))-\sin(1)) r^2+6 (l+q) r^3+(8+\sin(1)+l q (2+\sin(1))) r^4+(l+q) (1+\sin(1)) r^5+\sin(1) r^6$.
\end{enumerate}
\end{theorem}
\begin{proof}
\begin{enumerate}[(i)]
    \item Note that $x_1(0)=-\sin(1) <0$ and $x_1(1)=6(2 + d) (2 + q) (2 + s)>0$ and thus in view of the intermediate value theorem, there exists a root  of the  equation $x_1(r)=0$  in the interval $(0,1)$. Let  $r=\rho_1\in(0, 1)$ be the smallest root of the equation $x_1(r)=0$.
      For $|C-1| \le \sin(1),$ Cho \emph{et al.} \cite{cho} established the following inclusion:
      \begin{equation}\label{eq7.1}
      \{w \in \mathbb{C}: |w-C|<\sin(1)-|C-1|\} \subseteq \Omega_s,
      \end{equation}
      where $\Omega_s:=q_0(\Delta)$ is the image of the unit disc $\Delta$ under the mappings $q_0(z)=1+\sin(z)$.
      As the centre of the disc in \eqref{eq3.1} is 1,  by \eqref{eq7.1}, $f\in \mathcal{S}^*_{\sin}$ if
       \begin{equation}\label{eq7.2}
       \frac{r}{(1-r^2)}\left( \frac{dr^2+4r+d}{r^2+dr+1}+\frac{sr^2+4r+s}{r^2+sr+1}+\frac{q r^2+4r+q}{r^2+q r+1}\right)\leq\sin(1)
       \end{equation}
       which is equivalent to $f\in \mathcal{S}^*_{\sin}$ if  $x_1(r)\leq0$. Since $x_1(0)<0$ and $\rho_1$ is the smallest root of the equation $x_1(r)=0$,  $x_1(r)$ is an increasing function on $(0, \rho_1)$. In view of this  $f\in \mathcal{S}^*_{\sin}$ for $|z|=r\leq\rho_1$.
       For $u=6b-4c\geq0$, $v=4c-q\geq0$, using \eqref{eq7.2},   the function $f_1(z)$ defined for class $H^1_{b,c}$ in \eqref{eq0.6}  at $z=-\rho_1$, satisfies the  following equality
     \begin{align*}
    \left|\frac{zf_1'(z)}{f_1(z)}\right|&=\left|1+\frac{\rho_1}{(1-\rho_1^2)}\left( \frac{u\rho_1^2+4\rho_1+u}{\rho_1^2+u\rho_1+1}+\frac{v\rho_1^2+4\rho_1+v}{\rho_1^2+v\rho_1+1}+\frac{q \rho_1^2+4\rho_1+q}{\rho_1^2+q \rho_1+1}\right)\right|\\
    &=1+\sin{1}=q_0(1),
    \end{align*}
  which belongs to the boundary of region $\Omega_s$. Thus, the radius is sharp.
%
    \item A calculation shows that $x_2(0)=-\sin(1)<0$ and $x_2(1)= 6(2+m)(1+n)(2+q)>0$. By  intermediate value theorem, there exists a root $r\in(0, 1)$ of the equation $x_2(r) = 0$. Let $\rho_2\in(0, 1)$  be the smallest root of  equation $x_2(r)=0$ and $f\in H^2_{b,c}$. In view of \eqref{eq7.1} and the fact that  centre of the disc in \eqref{eq2.3} is 1,  $f\in \mathcal{S}^*_{\sin}$ if
       \begin{equation}\label{eq7.3}
       \frac{r}{(1-r^2)}\left( \frac{mr^2+4r+m}{r^2+mr+1}+\frac{nr^2+2r+n}{nr+1}+\frac{q r^2+4r+q}{r^2+q r+1}\right)\leq\sin(1)
       \end{equation}
    which is equivalent to $f\in \mathcal{S}^*_{\sin}$ if  $x_2(r)\leq0$. Since $x_2(0)<0$ and $\rho_2$ is the smallest root of the equation $x_2(r)=0$,  $x_2(r)$ is an increasing function on $(0, \rho_2)$. Thus, $f\in \mathcal{S}^*_{\sin}$ for $|z|=r\leq\rho_2$. To prove sharpness, consider the function $f_2$ defined in \eqref{eq1.3}. For $u=5b-3c\geq0$, $v=3c-q\geq0$ and $z:=-\rho_2$, it follows from \eqref{eq7.3} that
      \begin{align*}
      \left|\frac{zf_2'(z)}{f_2(z)}\right|&=\left|1+\frac{\rho_2}{(1-\rho_2^2)}\left( \frac{u\rho_2^2+4\rho_2+u}{\rho_2^2+u\rho_2+1}+\frac{v\rho_2^2+2\rho_2+v}{v\rho_2+1}+\frac{q \rho_2^2+4\rho_2+q}{\rho_2^2+q \rho_2+1}\right)\right|\\
      &=1+\sin(1)
      \end{align*} which illustrates the sharpness.

    \item It is easy to see that $x_3(0)=-\sin(1)<0$ and $x_3(1)=4(2+l)(2+q)>0$. By intermediate  value
theorem, there exists a root $r\in(0, 1)$ of the equation $x_3(r) = 0$. Let  $\rho_3\in(0, 1)$ be the smallest root of the equation $x_3(r)=0$.
Since  the centre of the disc in \eqref{eq3.5} is 1,  by \eqref{eq7.1}, $f\in \mathcal{S}^*_{\sin}$ if
       \begin{equation}\label{eq7.4}
       \frac{r}{1-r^2} \left(\frac{l r^2+4 r+l}{r^2+l r+1}+\frac{q r^2+4 r+q}{r^2+q r +1}\right)\leq\sin(1)
       \end{equation}
    which is equivalent to $f\in \mathcal{S}^*_{\sin}$ if  $x_3(r)\leq0$. Since $x_3(0)<0$ and $\rho_3$ is the smallest root of the equation $x_3(r)=0$,  $x_3(r)$ is an increasing function on $(0, \rho_3).$ This proves that $f\in \mathcal{S}^*_{\sin}$ for $|z|=r\leq\rho_3$.

\noindent
 The result is sharp for  function $f_3$ defined for the class $H^3_{b}$ in \eqref{eq0.10}.  At $z=-\rho_3$ and for $u=4b-q\geq0$,
 it follows from \eqref{eq7.4} that
\begin{align*}
\left|\frac{zf_3'(z)}{f_3(z)}\right| = \left|1+\frac{\rho_3}{1-\rho_3^2} \left(\frac{u \rho_3^2+4 \rho_3+u}{\rho_3^2+u \rho_3+1}+\frac{q \rho_3^2+4 \rho_3+q}{\rho_3^2+q \rho_3 +1}\right)\right|=1+\sin(1).
\end{align*}
\end{enumerate}
\end{proof}
\begin{remark}
 Substituting  $b=1$ and $c=1$ and $q=2$ in  Theorem \ref{th6}, we obtain the result \cite[Theorem 6, p. 13]{lecko}.
\end{remark}
\noindent
In 2015, Raina and Sok\'{o}l \cite{raina}  introduced the class $\mathcal{S}_{\leftmoon}^*=\mathcal{S}^*(z+\sqrt{1+z^2})$. Geometrically, a function $f \in \mathcal{S}_{\leftmoon}^*$ if and only if $zf'(z)/f(z)$ lies in the region bounded by the  lune shaped region $\{w \in \mathbb{C}: |w^2-1|< 2|w|\}.$
\begin{theorem}\label{th7}
 The $\mathcal{S}_{\leftmoon}^*$ radii for the classes $H^1_{b,c}$,  $H^2_{b,c}$ and $H^3_{b}$ are as follows:
\begin{enumerate}[(i)]
    \item For the class $H^1_{b,c}$, the sharp $\mathcal{S}_{\leftmoon}^*$ radius $\rho_1\in(0, 1)$ is the smallest  root  of the equation $x_1(r)=0$, where\\
   $x_1(r)=\sqrt{2}-2+(\sqrt{2}-1) (d+q+s) r+(2 (4+\sqrt{2})+\sqrt{2} (q s+d (q+s))) r^2+((9+\sqrt{2}) (q+s)+d (9+\sqrt{2}+(1+\sqrt{2}) q s)) r^3+8 (3+q s+d (q+s)) r^4+(-(-13+\sqrt{2}) (q+s)-d (-13+\sqrt{2}+(-5+\sqrt{2}) q s)) r^5+(-2 (-8+\sqrt{2})-(-4+\sqrt{2}) (q s+d (q+s))) r^6-(-3+\sqrt{2}) (d+q+s) r^7+(2-\sqrt{2}) r^8$.
    \item For the class $H^2_{b,c}$, the $\mathcal{S}_{\leftmoon}^*$ radius $\rho_2\in(0, 1)$ is the smallest  root  of the equation $x_2(r)=0$, where\\
    $x_2(r)=\sqrt{2}-2+(\sqrt{2}-1) (m+n+q) r+(8+\sqrt{2}+\sqrt{2} (n q+m (n+q))) r^2+(8 (m+q)+n (9+\sqrt{2}+(1+\sqrt{2}) m q)) r^3+(14-\sqrt{2}+6 m q-\sqrt{2} m q+8 n (m+q)) r^4+(-(-5+\sqrt{2}) (m+q)-n (-13+\sqrt{2}+(-5+\sqrt{2}) m q)) r^5-(-4+\sqrt{2}) (1+n (m+q)) r^6-(\sqrt{2}-3) n r^7$.
    \item  For the class $H^3_{b}$, the sharp $\mathcal{S}_{\leftmoon}^*$ radius $\rho_3\in(0, 1)$ is the smallest  root  of the equation $x_3(r)=0$, where\\
    $x_3(r)=\sqrt{2}-2+(\sqrt{2}-1) (l+q) r+(6+\sqrt{2}+\sqrt{2} l q) r^2+6 (l+q) r^3+(10-\sqrt{2}-(-4+\sqrt{2}) l q) r^4-(-3+\sqrt{2}) (l+q) r^5+(2-\sqrt{2}) r^6$.
\end{enumerate}
\end{theorem}
\begin{proof}
\begin{enumerate}[(i)]
    \item Note that $x_1(0)=\sqrt2-2 <0$ and $x_1(1)=6(2 + d) (2 + q) (2 + s)>0$ and thus in view of the intermediate value theorem, there exists a root  of the  equation $x_1(r)=0$  in the interval $(0,1)$. Let  $r=\rho_1\in(0, 1)$ be the smallest root of the equation $x_1(r)=0$.
      Gandhi and Ravichandran \cite[Lemma 2.1]{gandhi1} proved that for $\sqrt2-1<C\leq\sqrt2+1$,
      \begin{equation}\label{eq8.1}
     \{w \in \mathbb{C}: |w-C|<1-|\sqrt{2}-C|\} \subseteq \{w \in \mathbb{C}: |w^2-1|< 2|w|\}.
      \end{equation}
           As  the centre of the disc in \eqref{eq3.1} is 1,  by \eqref{eq8.1}, $f\in \mathcal{S}^*_{\leftmoon}$ if
       \begin{equation}\label{eq8.2}
       \frac{r}{(1-r^2)}\left( \frac{dr^2+4r+d}{r^2+dr+1}+\frac{sr^2+4r+s}{r^2+sr+1}+\frac{q r^2+4r+q}{r^2+q r+1}\right)\leq2-\sqrt2
       \end{equation}
       which is equivalent to $f\in \mathcal{S}^*_{\leftmoon}$ if  $x_1(r)\leq0$. Since $x_1(0)<0$ and $\rho_1$ is the smallest root of the equation $x_1(r)=0$,  $x_1(r)$ is an increasing function on $(0, \rho_1)$. In view of this  $f\in \mathcal{S}^*_{\leftmoon}$ for $|z|=r\leq\rho_1$.
       For $u=6b-4c\geq0$, $v=4c-q\geq0$, using \eqref{eq8.2},   the function $F_1(z)$ defined for the class $H^1_{b,c}$ in \eqref{eq0.8}  at $z=-\rho_1$, satisfies the  following equality
     \begin{align*}
   &\left|\left(\frac{zf_1'(z)}{f_1(z)}\right)^2-1\right|\\
    &=\left|\left(1-\frac{\rho_1}{(1-\rho_1^2)}\left( \frac{u\rho_1^2+4\rho_1+u}{\rho_1^2+u\rho_1+1}+\frac{v\rho_1^2+4\rho_1+v}{\rho_1^2+v\rho_1+1}+\frac{q \rho_1^2+4\rho_1+q}{\rho_1^2+q \rho_1+1}\right)\right)^2-1\right|\\
    &=|1-(2-\sqrt2)^2-1|=2(1-\sqrt2)=2\left|\frac{zf_1'(z)}{f_1(z)}\right|.
    \end{align*} Thus, the radius is sharp.

    \item A calculation shows that $x_2(0)=\sqrt2-2<0$ and $x_2(1)= 6(2+m)(1+n)(2+q)>0$. By  intermediate value theorem, there exists a root $r\in(0, 1)$ of the equation $x_2(r) = 0$. Let $\rho_2\in(0, 1)$  be the smallest root of the equation $x_2(r)=0$ and $f\in H^2_{b,c}$. In view of \eqref{eq8.1} and the fact centre of the disc in \eqref{eq2.3} is 1, $f\in \mathcal{S}^*_{\leftmoon}$ if
       \begin{equation*}
       \frac{r}{(1-r^2)}\left( \frac{mr^2+4r+m}{r^2+mr+1}+\frac{nr^2+2r+n}{nr+1}+\frac{q r^2+4r+q}{r^2+q r+1}\right)\leq2-\sqrt2
       \end{equation*}
    which is equivalent to $f\in \mathcal{S}^*_{\leftmoon}$ if  $x_2(r)\leq0$. Since $x_2(0)<0$ and $\rho_2$ is the smallest root of the equation $x_2(r)=0$,  $x_2(r)$ is an increasing function on $(0, \rho_2)$. Thus, $f\in \mathcal{S}^*_{\leftmoon}$ for $|z|=r\leq\rho_2$.

\item It is easy to see that $x_3(0)=\sqrt2-2<0$ and $x_3(1)=4(2+l)(2+q)>0$. By intermediate value
theorem, there exists a root $r\in(0, 1)$ of the equation $x_3(r) = 0$. Let  $\rho_3\in(0, 1)$ be the smallest root of the equation $x_3(r)=0$.
Since  the centre of the disc in \eqref{eq3.5} is 1,  by \eqref{eq8.1}, $f\in \mathcal{S}^*_{\leftmoon}$ if
       \begin{equation}\label{eq8.4}
       \frac{r}{1-r^2} \left(\frac{l r^2+4 r+l}{r^2+l r+1}+\frac{q r^2+4 r+q}{r^2+q r +1}\right)\leq2-\sqrt2
       \end{equation}
    which is equivalent to $f\in \mathcal{S}^*_{\leftmoon}$ if  $x_3(r)\leq0$. Since $x_3(0)<0$ and $\rho_3$ is the smallest root of the equation $x_3(r)=0$,  $x_3(r)$ is an increasing function on $(0, \rho_3).$ This proves that $f\in \mathcal{S}^*_{\leftmoon}$ for $|z|=r\leq\rho_3$.

 The result is sharp for the function $F_3$ defined for the class $H^3_{b}$ in \eqref{eq2.5}.  At $z=-\rho_3$, and for $u=4b-q\geq0$,
 it follows from \eqref{eq8.4} that
\begin{align*}
\left|\left(\frac{zF_3'(z)}{F_3(z)}\right)^2-1\right|& = \left|\left(1-\frac{\rho_3}{1-\rho_3^2} \left(\frac{u \rho_3^2+4 \rho_3+u}{\rho_3^2+u \rho_3+1}+\frac{q \rho_3^2+4 \rho_3+q}{\rho_3^2+q \rho_3 +1}\right)\right)^2-1\right|\\&=2\left|\frac{zF_3'(z)}{F_3(z)}\right|. \qedhere
\end{align*}
\end{enumerate}
\end{proof}
\begin{remark}
 For $b=1$ and $c=1$ and $q=2$,  Theorem \ref{th7} yields the result \cite[Theorem 7, p. 14]{lecko}.
\end{remark}

\noindent
Kumar \emph{et al.}\cite{kumar} introduced the class of starlike functions,  defined by $\mathcal{S}^*_R=\mathcal{S}^*(\psi(z))$, consisting of functions associated with a rational function $\psi(z)=1+{z(k+z)}/{(k(k-z))}$, where $k=\sqrt{2}+1.$ The following theorem yields  $\mathcal{S}^*_R$  radii.
\begin{theorem}\label{th8}
The  $\mathcal{S}_{R}^*$ radii for the classes $H^1_{b,c}$,  $H^2_{b,c}$ and $H^3_{b}$ are as follows:
\begin{enumerate}[(i)]
    \item For the class $H^1_{b,c}$, the sharp $\mathcal{S}_{R}^*$ radius $\rho_1\in(0, 1)$ is the smallest  root  of the equation $x_1(r)=0$, where\\
    $x_1(r)=(2 \sqrt{2}-3)+2 (\sqrt{2}-1) (d+q+s) r+(6+4\sqrt{2}+(2 \sqrt{2}-1) (q s+d (q+s))) r^2+2 ((4+\sqrt{2}) (q+s)+d (4+\sqrt{2}+\sqrt{2} q s)) r^3+8 (3+q s+d (q+s)) r^4-2 ((-7+\sqrt{2}) (q+s)+d (-7+\sqrt{2}+(-3+\sqrt{2}) q s)) r^5+(18-4 \sqrt{2}+(5-2\sqrt{2}) (q s+d (q+s))) r^6-2 ((-2+\sqrt{2}) (d+q+s)) r^7+(3-2 \sqrt{2}) r^8$.
    \item For the class $H^2_{b,c}$, the  $\mathcal{S}_{R}^*$ radius $\rho_2\in(0, 1)$ is the smallest  root  of the equation $x_2(r)=0$, where\\
    $x_2(r)=(2\sqrt2-3)+2 (\sqrt{2}-1) (m+n+q) r+(7+2 \sqrt{2}+(2 \sqrt{2}-1) (n q+m (n+q))) r^2+(8 (m+q)+2 n (4+\sqrt{2}+\sqrt{2} m q)) r^3+(15-2 \sqrt{2}+7 m q-2 \sqrt{2} m q+8 n (m+q)) r^4-2 ((-7+\sqrt{2}) n+(-3+\sqrt{2}) q+(-3+\sqrt{2}) m (1+n q)) r^5+(5-2 \sqrt{2}) (1+n (m+q)) r^6-2 ((-2+\sqrt{2}) n) r^7$.
    \item For the class $H^3_{b}$, the sharp $\mathcal{S}_{R}^*$ radius $\rho_3\in(0, 1)$ is the smallest  root  of the equation $x_3(r)=0$, where\\
    $x_3(r)=(2\sqrt{2}-3)+2 (\sqrt{2}-1) (l+q) r+(5+2 \sqrt{2}+(-1+2\sqrt{2}) l q) r^2+6 (l+q) r^3+(11-2\sqrt{2}+(5-2 \sqrt{2}) l q) r^4-2 ((-2+\sqrt{2}) (l+q)) r^5+(3-2\sqrt{2}) r^6.$
\end{enumerate}
\end{theorem}
\begin{proof}
\begin{enumerate}[(i)]
    \item Note that $x_1(0)=2 \sqrt{2}-3<0$ and $x_1(1)=6(2 + d) (2 + q) (2 + s)>0$ and thus in view of the intermediate value theorem, there exists a root  of the  equation $x_1(r)=0$  in the interval $(0,1)$. Let  $r=\rho_1\in(0, 1)$ be the smallest root of the equation $x_1(r)=0$.
      For $2(\sqrt{2}-1)<C \le\sqrt{2},$ Kumar \emph{et al.}\cite{kumar}  proved that
      \begin{equation}\label{eq9.1}
     \{w \in \mathbb{C}: |w-C|<C-2(\sqrt{2}-1) \} \subseteq \psi(\Delta).
      \end{equation}
           As the centre of the disc in \eqref{eq3.1} is 1,  by \eqref{eq9.1} $f\in \mathcal{S}^*_{\leftmoon}$ if
       \begin{equation}\label{eq9.2}
       \frac{r}{(1-r^2)}\left( \frac{dr^2+4r+d}{r^2+dr+1}+\frac{sr^2+4r+s}{r^2+sr+1}+\frac{q r^2+4r+q}{r^2+q r+1}\right)\leq3-2\sqrt2
       \end{equation}
       which is equivalent to $f\in \mathcal{S}^*_{R}$ if  $x_1(r)\leq0$. Since $x_1(0)<0$ and $\rho_1$ is the smallest root of the equation $x_1(r)=0$,  $x_1(r)$ is an increasing function on $(0, \rho_1)$. In view of this  $f\in \mathcal{S}^*_{R}$ for $|z|=r\leq\rho_1$.
       For $u=6b-4c\geq0$, $v=4c-q\geq0$, using \eqref{eq9.2},   the function $F_1(z)$ defined for the class $H^1_{b,c}$ in \eqref{eq0.8}  at $z=-\rho_1$, satisfies the  following equality
     \begin{align*}
    \left|\frac{zF_1'(z)}{F_1(z)}\right|
    &=\left|1-\frac{\rho_1}{(1-\rho_1^2)}\left( \frac{u\rho_1^2+4\rho_1+u}{\rho_1^2+u\rho_1+1}+\frac{v\rho_1^2+4\rho_1+v}{\rho_1^2+v\rho_1+1}+\frac{q \rho_1^2+4\rho_1+q}{\rho_1^2+q \rho_1+1}\right)\right|\\
    &=2(\sqrt2-1)= \psi(-1).
    \end{align*}
     Thus, the radius is sharp.

    \item A calculation shows that $x_2(0)=2\sqrt2-3<0$ and $x_2(1)= 6(2+m)(l+n)(2+q)>0$. By  intermediate value theorem, there exists a root $r\in(0, 1)$ of the equation $x_2(r) = 0$. Let $\rho_2\in(0, 1)$  be the smallest root of the equation $x_2(r)=0$ and $f\in H^2_{b,c}$. In view of \eqref{eq9.1} and the fact that   the centre of the disc in \eqref{eq2.3} is 1, $f\in \mathcal{S}^*_{R}$ if
       \begin{equation*}
       \frac{r}{(1-r^2)}\left( \frac{mr^2+4r+m}{r^2+mr+1}+\frac{nr^2+2r+n}{nr+1}+\frac{q r^2+4r+q}{r^2+q r+1}\right)\leq3-2\sqrt2
       \end{equation*}
    which is equivalent to $f\in \mathcal{S}^*_{R}$ if  $x_2(r)\leq0$. Since $x_2(0)<0$ and $\rho_2$ is the smallest root of the equation $x_2(r)=0$,  $x_2(r)$ is an increasing function on $(0, \rho_2)$. Thus, $f\in \mathcal{S}^*_{R}$ for $|z|=r\leq\rho_2$.

    \item It is easy to see that $x_3(0)=2\sqrt2-3<0$ and $x_3(1)=4(2+l)(2+q)>0$. By intermediate value
theorem, there exists a root $r\in(0, 1)$ of the equation $x_3(r) = 0$. Let  $\rho_3\in(0, 1)$ be the smallest root of the equation $x_3(r)=0$.
Since  the centre of the disc in \eqref{eq3.5} is 1,  by \eqref{eq9.1}, $f\in \mathcal{S}^*_{R}$ if
       \begin{equation}\label{eq9.4}
       \frac{r}{1-r^2} \left(\frac{l r^2+4 r+l}{r^2+l r+1}+\frac{q r^2+4 r+q}{r^2+q r +1}\right)\leq3-2\sqrt2
       \end{equation}
    which is equivalent to $f\in \mathcal{S}^*_{R}$ if  $x_3(r)\leq0$. Since $x_3(0)<0$ and $\rho_3$ is the smallest root of the equation $x_3(r)=0$,  $x_3(r)$ is an increasing function on $(0, \rho_3).$ This proves that $f\in \mathcal{S}^*_{R}$ for $|z|=r\leq\rho_3$.

 The result is sharp for the function $F_3$ defined for the class $H^3_{b}$ in \eqref{eq2.5}.  At $z=-\rho_3$, and for $u=4b-q\geq0$,
 it follows from \eqref{eq9.4} that
\begin{align*}
\left|\frac{zF_3'(z)}{F_3(z)}\right|& = \left|1-\frac{\rho_3}{1-\rho_3^2} \left(\frac{u \rho_3^2+4 \rho_3+u}{\rho_3^2+u \rho_3+1}+\frac{q \rho_3^2+4 \rho_3+q}{\rho_3^2+q \rho_3 +1}\right)\right|= \psi(-1). \qedhere
\end{align*}
\end{enumerate}
\end{proof}
\begin{remark}
 Substituting  $b=1$ and $c=1$ and $q=2$ in  Theorem \ref{th8}, we get the result \cite[Theorem 8, p. 15]{lecko}.
\end{remark}

\noindent
\noindent
In 2020, Wani and Swaminathan \cite[Lemma 2.2]{wani} introduced the class $\mathcal{S}^*_{Ne}=\mathcal{S}^*(1+z-(z^3/3))$ consisting of functions associated  with a nephroid. So, a function $f\in\mathcal{S}^*_{Ne}$ if and only if $zf'/f$  maps the open unit disc $\Delta$ onto the interior of a two cusped kidney shaped curve $\Omega_{Ne}:=\{u+iv:((u-1)^2+v^2-4/9)^3-4v^2/3 <0\}.$
\begin{theorem}\label{th9}
The sharp $\mathcal{S}_{Ne}^*$ radii for the classes $H^1_{b,c}$,  $H^2_{b,c}$ and $H^3_{b}$ are as follows:
\begin{enumerate}[(i)]
    \item For the class $H^1_{b,c}$, the sharp $\mathcal{S}_{Ne}^*$ radius $\rho_1\in(0, 1)$ is the smallest  root  of the equation $x_1(r)=0$, where\\
    $x_1(r)=-2+(d+q+s) r+4 (8+q s+d (q+s)) r^2+(31 (q+s)+d (31+7 q s)) r^3+24 (3+q s+d (q+s)) r^4+(35 (q+s)+d (35+11 q s)) r^5+8 (5+q s+d (q+s)) r^6+5 (d+q+s) r^7+2 r^8$.
    \item For the class $H^2_{b,c}$, the sharp $\mathcal{S}_{Ne}^*$ radius $\rho_2\in(0, 1)$ is the smallest  root  of the equation $x_2(r)=0$, where\\
    $x_2(r)= -2+(m+n+q) r+4 (7+n q+m (n+q)) r^2+(31 n+24 q+m (24+7 n q)) r^3+(38+24 m n+14 m q+24 n q) r^4+(35 n+11 q+11 m (1+n q)) r^5+8 (1+n (m+q)) r^6+5 n r^7.$
    \item For the class $H^3_{b}$, the sharp $\mathcal{S}_{Ne}^*$ radius $\rho_3\in(0, 1)$ is the smallest  root  of the equation $x_3(r)=0$, where\\
    $x_3(r)=-2+(2+l) r+(22+8 l) r^2+18 (2+l) r^3+(26+16 l) r^4+5 (2+l) r^5+2 r^6$.
\end{enumerate}
\end{theorem}
\begin{proof}
\begin{enumerate}[(i)]
    \item Note that $x_1(0)=-2<0$ and $x_1(1)=18(2 + d) (2 + q) (2 + s)>0$ and thus in view of the intermediate value theorem, there exists a root  of the  equation $x_1(r)=0$  in the interval $(0,1)$. Let  $r=\rho_1\in(0, 1)$ be the smallest root of the equation $x_1(r)=0$.
 For $1 \le C < 5/3,$ Wani and Swaminathan \cite[Lemma 2.2]{wani} had proved that
      \begin{equation}\label{eq10.1}
     \{w \in \mathbb{C}: |w-C|< 5/3-C\} \subseteq \Omega_{Ne}.
      \end{equation}
           As  the centre of the disc in \eqref{eq3.1} is 1,  by \eqref{eq10.1}, $f\in \mathcal{S}^*_{Ne}$ if
       \begin{equation}\label{eq10.2}
       \frac{r}{(1-r^2)}\left( \frac{dr^2+4r+d}{r^2+dr+1}+\frac{sr^2+4r+s}{r^2+sr+1}+\frac{q r^2+4r+q}{r^2+q r+1}\right)\leq\frac{2}{3}
       \end{equation}
       which is equivalent to $f\in \mathcal{S}^*_{Ne}$ if  $x_1(r)\leq0$. Since $x_1(0)<0$ and $\rho_1$ is the smallest root of the equation $x_1(r)=0$,  $x_1(r)$ is an increasing function on $(0, \rho_1)$. In view of this  $f\in \mathcal{S}^*_{Ne}$ for $|z|=r\leq\rho_1$.
       For $u=6b-4c\geq0$, $v=4c-q\geq0$, using \eqref{eq10.2},   the function $f_1(z)$ defined for the class $H^1_{b,c}$ in \eqref{eq0.6}  at $z=-\rho_1$, satisfies the  following equality
     \begin{align*}
    \left|\frac{zf_1'(z)}{f_1(z)}\right|
    &=\left|1+\frac{\rho_1}{(1-\rho_1^2)}\left( \frac{u\rho_1^2+4\rho_1+u}{\rho_1^2+u\rho_1+1}+\frac{v\rho_1^2+4\rho_1+v}{\rho_1^2+v\rho_1+1}+\frac{q \rho_1^2+4\rho_1+q}{\rho_1^2+q \rho_1+1}\right)\right|\\
    &= \frac{5}{3}.
    \end{align*}
    which belongs to the boundary of the region $\Omega_{Ne}$. Thus, the radius is sharp.

    \item A calculation shows that $x_2(0)=-2<0$ and $x_2(1)= 18(2+m)(l+n)(2+q)>0$. By  intermediate value theorem, there exists a root $r\in(0, 1)$ of the equation $x_2(r) = 0$. Let $\rho_2\in(0, 1)$  be the smallest root of the equation $x_2(r)=0$ and $f\in H^2_{b,c}$. In view of \eqref{eq10.1} and the fact that centre of the disc in \eqref{eq2.3} is 1, $f\in \mathcal{S}^*_{Ne}$ if
       \begin{equation}\label{eq10.3}
       \frac{r}{(1-r^2)}\left( \frac{mr^2+4r+m}{r^2+mr+1}+\frac{nr^2+2r+n}{nr+1}+\frac{q r^2+4r+q}{r^2+q r+1}\right)\leq\frac{2}{3}
       \end{equation}
    which is equivalent to $f\in \mathcal{S}^*_{Ne}$ if  $x_2(r)\leq0$. Since $x_2(0)<0$ and $\rho_2$ is the smallest root of the equation $x_2(r)=0$,  $x_2(r)$ is an increasing function on $(0, \rho_2)$. Thus, $f\in \mathcal{S}^*_{Ne}$ for $|z|=r\leq\rho_2$. To prove the sharpness, consider the function $f_2$ defined in \eqref{eq1.3}. For $u=5b-3c\geq0$, $v=3c-q\geq0$ and $z=-\rho_2$, it follows from \eqref{eq10.3} that
      \begin{align*}
      \left|\frac{zf_2'(z)}{f_2(z)}\right|&=\left|1+\frac{\rho_2}{(1-\rho_2^2)}\left( \frac{u\rho_2^2+4\rho_2+u}{\rho_2^2+u\rho_2+1}+\frac{v\rho_2^2+2\rho_2+v}{v\rho_2+1}+\frac{q \rho_2^2+4\rho_2+q}{\rho_2^2+q \rho_2+1}\right)\right|\\
      &= \frac{5}{3}
      \end{align*} which illustrates the sharpness.

     \item It is easy to see that $x_3(0)=-2<0$ and $x_3(1)=12(2+l)(2+q)>0$. By intermediate  value
theorem, there exists a root $r\in(0, 1)$ of the equation $x_3(r) = 0$. Let  $\rho_3\in(0, 1)$ be the smallest root of the equation $x_3(r)=0$.
Since  the centre of the disc in \eqref{eq3.5} is 1,  by \eqref{eq10.1}, $f\in \mathcal{S}^*_{Ne}$ if
       \begin{equation}\label{eq10.4}
       \frac{r}{1-r^2} \left(\frac{l r^2+4 r+l}{r^2+l r+1}+\frac{q r^2+4 r+q}{r^2+q r +1}\right)\leq\frac{2}{3}
       \end{equation}
    which is equivalent to $f\in \mathcal{S}^*_{Ne}$ if  $x_3(r)\leq0$. Since $x_3(0)<0$ and $\rho_3$ is the smallest root of the equation $x_3(r)=0$,  $x_3(r)$ is an increasing function on $(0, \rho_3).$ This proves that $f\in \mathcal{S}^*_{Ne}$ for $|z|=r\leq\rho_3$.

 The result is sharp for the function $f_3$ defined for the class $H^3_{b}$ in \eqref{eq0.10}.  At $z=-\rho_3$, and for $u=4b-q\geq0$,
 it follows from \eqref{eq10.4} that
\begin{align*}
\left|\frac{zf_3'(z)}{f_3(z)}\right|& = \left|1-\frac{\rho_3}{1-\rho_3^2} \left(\frac{u \rho_3^2+4 \rho_3+u}{\rho_3^2+u \rho_3+1}+\frac{q \rho_3^2+4 \rho_3+q}{\rho_3^2+q \rho_3 +1}\right)\right|= \frac{5}{3}.
\end{align*}
\end{enumerate}
\end{proof}
\begin{remark}
 For $b=1$ and $c=1$ and $q=2$, Theorem \ref{th9} reduces to  the result \cite[Theorem 10, p. 18]{lecko}.
\end{remark}

\noindent
In 2020,  the class $\mathcal{S}_{SG}^*=\mathcal{S}^*(2/(1+e^{-z}))$  that maps the open unit disc $\Delta$ onto a domain $\Delta_{SG}:=\{w \in \mathbb{C}: |\log(w/(2-w))|<1\}$   was introduced by   Goel and Kumar\cite{goel}.   Some results for the class  $\mathcal{S}_{SG}^*$ can be seen in \cite{jain}. The following theorem gives the sharp $\mathcal{S}_{SG}^*$ radii.
\begin{theorem}\label{th10}
The sharp $\mathcal{S}_{SG}^*$ radii for the classes $H^1_{b,c}$,  $H^2_{b,c}$ and $H^3_{b}$ are as follows:
\begin{enumerate}[(i)]
    \item For the class $H^1_{b,c}$, the sharp $\mathcal{S}_{SG}^*$ radius $\rho_1\in(0, 1)$ is the smallest  root  of the equation $x_1(r)=0$, where\\
    $x_1(r)= (1-e)+2 (d+q+s) r+(2 (7+5 e)+(3+e) (q s+d (q+s))) r^2+(2 (6+5 e) (q+s)+2 d (6+5 e+(2+e) q s)) r^3+8 (1+e) (3+q s+d (q+s)) r^4+2 ((5+6 e) (q+s)+d (5+q s+2 e (3+q s))) r^5+(2 (5+7 e)+(1+3 e) (q s+d (q+s))) r^6+2 e (d+q+s) r^7+(-1+e) r^8.$
    \item For the class $H^2_{b,c}$, the sharp $\mathcal{S}_{SG}^*$ radius $\rho_2\in(0, 1)$ is the smallest  root  of the equation $x_2(r)=0$, where\\
    $ x_2(r)=(1-e)+2 (m+n+q) r+(11+9 e+(3+e) (n q+m (n+q))) r^2+2 (4 (1+e) m+(6+5 e) n+4 (1+e) q+(2+e) m n q) r^3+(11+8 m n+3 m q+8 n q+e (13+8 m n+5 m q+8 n q)) r^4+2 ((5+6 e) n+q+2 e q+(1+2 e) m (1+n q)) r^5+(1+3 e) (1+n (m+q)) r^6+2 e n r^7.$
    \item For the class $H^3_{b}$, the sharp $\mathcal{S}_{SG}^*$ radius $\rho_3\in(0, 1)$ is the smallest  root  of the equation $x_3(r)=0$, where\\
    $x_3(r)=(1-e)+2 (l+q) r+(9+7 e+(3+e) l q) r^2+6 (1+e) (l+q) r^3+(7+l q+3 e (3+l q)) r^4+2 e (l+q) r^5+(-1+e) r^6.$
\end{enumerate}
\end{theorem}
\begin{proof}
\begin{enumerate}[(i)]
    \item Note that $x_1(0)=1-e<0$ and $x_1(1)=6(1+e)(2 + d) (2 + q) (2 + s)>0$ and thus in view of the intermediate value theorem, there exists a root  of the  equation $x_1(r)=0$  in the interval $(0,1)$. Let  $r=\rho_1\in(0, 1)$ be the smallest root of the equation $x_1(r)=0$.
For $2/(1+e)<C<2e/(1+e),$ Goel and Kumar\cite{goel}  proved the following inclusion:
           \begin{equation}\label{eq11.1}
         \{w \in \mathbb{C}: |w-C|<r_{SG}\} \subset \Delta_{SG},\,\, \text{where} \,\,r_{SG}=\left(\frac{e-1}{e+1}\right)-|C-1|.
      \end{equation}
           As  the centre of the disc in \eqref{eq3.1} is 1,  by \eqref{eq11.1}, $f\in \mathcal{S}^*_{SG}$ if
       \begin{equation}\label{eq11.2}
       \frac{r}{(1-r^2)}\left( \frac{dr^2+4r+d}{r^2+dr+1}+\frac{sr^2+4r+s}{r^2+sr+1}+\frac{q r^2+4r+q}{r^2+q r+1}\right)\leq \frac{e-1}{e+1}
       \end{equation}
       which is equivalent to $f\in \mathcal{S}^*_{SG}$ if  $x_1(r)\leq0$. Since $x_1(0)<0$ and $\rho_1$ is the smallest root of the equation $x_1(r)=0$,  $x_1(r)$ is an increasing function on $(0, \rho_1)$. In view of this  $f\in \mathcal{S}^*_{SG}$ for $|z|=r\leq\rho_1$.
       For $u=6b-4c\geq0$, $v=4c-q\geq0$, using \eqref{eq11.2},   the function $f_1(z)$ defined for the class $H^1_{b,c}$ in \eqref{eq0.6}  at $z=-\rho_1$, satisfies the  following equality
     \begin{align}\label{eq11.5}
    \left|\log\left(
    \frac{\frac{zf_1'(z)}{f_1(z)}}{2-\frac{zf_1'(z)}{f_1(z)}}\right)\right|
    =\left|\log\left(
    \frac{1+\frac{e-1}{e+1}}{2-\left(1+\frac{e-1}{e+1}\right)}\right)\right|=\left|\log(e)\right|=1.
    \end{align}
     Thus, the radius is sharp.


    \item A calculation shows that $x_2(0)=1-e<0$ and $x_2(1)= 6(1+e)(2+m)(l+n)(2+q)>0$. By  intermediate value theorem, there exists a root $r\in(0, 1)$ of the equation $x_2(r) = 0$. Let $\rho_2\in(0, 1)$  be the smallest root of the equation $x_2(r)=0$ and $f\in H^2_{b,c}$. In view of \eqref{eq11.1} and the fact that  the centre of the disc in \eqref{eq2.3} is 1, $f\in \mathcal{S}^*_{SG}$ if
       \begin{equation}\label{eq11.3}
       \frac{r}{(1-r^2)}\left( \frac{mr^2+4r+m}{r^2+mr+1}+\frac{nr^2+2r+n}{nr+1}+\frac{q r^2+4r+q}{r^2+q r+1}\right)\leq \frac{e-1}{e+1}
       \end{equation}
    which is equivalent to $f\in \mathcal{S}^*_{SG}$ if  $x_2(r)\leq0$. Since $x_2(0)<0$ and $\rho_2$ is the smallest root of the equation $x_2(r)=0$,  $x_2(r)$ is an increasing function on $(0, \rho_2)$. Thus, $f\in \mathcal{S}^*_{SG}$ for $|z|=r\leq\rho_2$. To prove the sharpness, consider the function $f_2$ defined in \eqref{eq1.3}. For $u=5b-3c\geq0$, $v=3c-q\geq0$ and $z=-\rho_2$, the similar calculations  as in \eqref{eq11.5} together with \eqref{eq11.3} yields the result is sharp.

    \item It is easy to see that $x_3(0)=1-e<0$ and $x_3(1)=4(1+e)(2+l)(2+q)>0$. By intermediate value
theorem, there exists a root $r\in(0, 1)$ of the equation $x_3(r) = 0$. Let  $\rho_3\in(0, 1)$ be the smallest root of the equation $x_3(r)=0$.
Since  the centre of the disc in \eqref{eq3.5} is 1,  by \eqref{eq11.1}, $f\in \mathcal{S}^*_{SG}$ if
       \begin{equation*}
       \frac{r}{1-r^2} \left(\frac{l r^2+4 r+l}{r^2+l r+1}+\frac{q r^2+4 r+q}{r^2+q r +1}\right)\leq\frac{e-1}{1+e}
       \end{equation*}
    which is equivalent to $f\in \mathcal{S}^*_{SG}$ if  $x_3(r)\leq0$. Since $x_3(0)<0$ and $\rho_3$ is the smallest root of the equation $x_3(r)=0$,  $x_3(r)$ is an increasing function on $(0, \rho_3).$ This proves that $f\in \mathcal{S}^*_{SG}$ for $|z|=r\leq\rho_3$.

  At $z=-\rho_3$, and for $u=4b-q\geq0$, a calculation as in part(i) shows that the result is sharp for the function $F_3$ defined for the class $H^3_{b}$ in \eqref{eq2.5}
\end{enumerate}
    \end{proof}
\begin{remark}
Putting  $b=1$ and $c=1$ and $q=2$ in  Theorem \ref{th10}, we obtain the result \cite[Theorem 11, p. 19]{lecko}.
\end{remark}

\noindent
\section*{Declaration}
\noindent
The authors declare that they have no conflict of interest.

\end{document}